\documentclass{article}
\usepackage[utf8]{inputenc}

\title{Malliavin-Stein method for the multivariate compound Hawkes process}
\author{
	  Mahmoud Khabou\footnote{INSA de Toulouse, IMT UMR CNRS 5219, Universit\'e de Toulouse, 135 avenue de Rangueil 31077 Toulouse Cedex 4 France. \; Email: \texttt{mahmoud.khabou@insa-toulouse.fr}} 
}
\usepackage{ dsfont }
\usepackage{graphicx}
\usepackage{amsmath}
\usepackage{amssymb}
\usepackage{algorithm}
\usepackage{algorithmic}
\usepackage{amsthm}
\usepackage{hyperref}
\usepackage{ stmaryrd }
\usepackage{cancel}
\usepackage{graphicx}
\usepackage{geometry}
\usepackage{systeme}
\usepackage{mathtools}
\DeclareMathOperator{\Hessian}{Hess}

\newtheorem{remark}{Remark}
\newtheorem{theorem}{Theorem}[section]

\newtheorem{definition}{Definition}[section]
\newtheorem{corollary}{Corollary}[theorem]
\newtheorem{lemma}[theorem]{Lemma}
\newtheorem{prop}{Proposition}[section] 
\newtheorem{defprop}{Definition and Proposition}[section] 
\usepackage{geometry}
\newcommand{\eps}{\varepsilon}
\def\E{\mathbb{E}}
\def\P{\mathbb{P}}
\def\real{\mathbb{R}}

\def\F{\mathcal{F}}
\def\1{\textbf{1}}

\def\R{\mathbb R}
\def\N{\mathbb N}
\def\E{\mathbb E}
\def\F{{\cal F}}

\def\d{{\mathrm d}}

\def\P{\mathbb P}

\newtheorem{assumption}{Assumption}

\DeclareMathOperator{\diag}{diag}  

\begin{document}
	
	\maketitle
	\begin{abstract}
		\noindent
    In this paper, we provide upper bounds on the $d_2$ distance between a large class of functionals of a multivariate compound Hawkes process and a given Gaussian vector. This is proven using Malliavin's calculus defined on an underlying Poisson embedding. The upper bound is then used to infer the speed of convergence of Central Limit Theorems  for the multivariate compound Hawkes process with exponential kernels as the observation time $T$ goes to infinity.
	\end{abstract}
	\section{Introduction}
	Hawkes processes have been used to model events that exhibit self-exciting properties. Initially introduced in 1971 \cite{hawkes1971spectra} to model seismic activities, Hawkes processes became popular in other fields like credit risk \cite{errais2010affine} or micro-structure in finance \cite{bacry2013modelling}. In recent years, the need to model systems with interacting components has been growing, that is how the multivariate Hawkes process emerged. In addition to its self-exciting properties, the multivariate Hawkes process can model mutually-exciting phenomena and it has been used in a wide array of fields such that neuro-science \cite{locherbach2017large}, \cite{reynaudbouret:hal-00866823}, social networks \cite{achab2017uncovering}, cyber-security \cite{bessyroland:hal-02546343} and financial econometrics \cite{embrechts_liniger_lin_2011}.\\
	
	In this paper, we consider the multivariate compound process (or total loss process) $(\boldsymbol{ L}_t)_{t\geq 0}$ defined  component-wise at a time $t$ as 
		\begin{equation}
\left\lbrace
\begin{array}{l}
L^1_t = \sum_{k=1}^{H^1_t} Y^1_k,\quad (Y^1_k)_{k=1,\cdots} \text { are } i.i.d ,\\\\
\vdots \\\\
L^d_t = \sum_{k=1}^{H^d_t} Y^d_k,\quad (Y^d_k)_{k=1,\cdots} \text { are } i.i.d ,
\end{array}
\right.
\end{equation}
	where $(Y^j_k)_{k\in \mathbb N^*, j=1,\cdots,d}$ are  \textit{i.i.d} random variables and $\boldsymbol{ H}_t=\left (H^1_t,\cdots,H^d_t \right )$ is a family of $d$ point processes each having an intensity $(\lambda_t^1,\cdots,\lambda_t^d)$ such that 
	$$\lambda_t^i \d t= \E \left [ H_{t+\d t}^i-H_{t}^i| \mathcal F_{t^-}\right ],\quad i=1,\cdots,d.$$
	The process $\boldsymbol{ L}$ is called a multivariate compound Hawkes process if its intensity $\boldsymbol{ \lambda}$ follows the dynamics
	
		\begin{equation}
\left\lbrace
\begin{array}{l}
\lambda^1_t = \mu^1+\int_{[0,t)} \sum_{k=1}^d \Phi_{1k}(t-s)\d L_s^k,\\\\
\vdots \\\\
\lambda^d_t = \mu^d+\int_{[0,t)} \sum_{k=1}^d \Phi_{dk}(t-s)\d L_s^k,
\end{array}
\right.
\end{equation}
where $\boldsymbol{ \mu}=(\mu^1,\cdots,\mu^d)\in \R_+^d$ plays the role of a baseline intensity and $\Phi = \left ( \Phi_{ik} \right )_{i,k=1,\cdots,d} $ is a family of non-negative integrable kernels on $\R_+$.\\
This dynamics can also be expressed under the following matrix form 

$$\boldsymbol{ \lambda}_t=\boldsymbol{ \mu}+ \int_{[0,t)}\Phi(t-s) \d \boldsymbol{ L}_s.$$
	In some applications, the question about the Hawkes process' longtime behaviour naturally arises. In the special case $Y^j_k \equiv 1$, and under the condition that the spectral radius of the matrix $S:= \|\Phi\|_1$ is strictly less than one ($\rho(S)<1$), Bacry \textit{et al.} \cite{bacry2013some}
	proved the following martingale Central Limit Theorem \textit{(CLT)} for $(\boldsymbol{ L}_t)_{t\geq 0}=(\boldsymbol{ H}_t)_{t\geq 0}$
	$$\left (\frac{\boldsymbol{ H}_{Tv}-\int_0^{Tv} \boldsymbol{ \lambda}_t \d t}{\sqrt T}\right )_{v\in[0,1]}\underset{T \rightarrow +\infty}{\Longrightarrow} \left ( \Sigma^{1/2}  W_v\right )_{v\in [0,1]}$$
	where the convergence takes place in law for the Skorokhod topology and where $\Sigma$ is a diagonal matrix that depends only on $\boldsymbol {\mu}$ and $S$. From this martingale result, they derived the alternative CLT
	$$\left (\frac{\boldsymbol{ H}_T-\int_0^T \E [\boldsymbol{ \lambda}_t] \d t}{\sqrt T}\right )_{v\in [0,1]}\underset{T \rightarrow +\infty}{\Longrightarrow} \left ( \left (I_d -S \right )^{-1}\Sigma^{1/2} W_v \right )_{v\in [0,1]}.$$
	Until very recently, the quantification of the speed of this convergence has not been thoroughly studied. In the univariate $1$-marginal case $(v=1)$, Hillairet \textit{et al.} \cite{hillairet2021malliavinstein} have found a Berry-Ess\'een type bound on the Wassestein distance between the normalized one-dimensional Hawkes process and its Gaussian limit
	$$d_W\left (\frac{H_T-\int_0^T \E[\lambda_t] \d t}{\sqrt T},  \mathcal N(0,\tilde\sigma^2)\right )=O\left (\frac{1}{\sqrt T} \right ),\quad \tilde \sigma^2 =\frac{\mu}{(1-\|\phi\|_1)^3}$$
	in case the kernel is an exponential $\phi(u)=\alpha e^{-\beta u}$ (with $0<\alpha < \beta$) or an Erlang function $\phi(u)=\alpha ue^{-\beta u}$ (with $0<\alpha < \beta^2$). This result has been derived thanks to an approach introduced by Nourdin and Peccati \cite{nourdin:hal-01314406} which combines Malliavin's calculus with Stein's method.\\
	
	In this article we prove a generalization of the quantification result to the multivariate case using a version of Malliavin's calculus that is adapted to higher dimensions. Following the lines of \cite{giovanni2010multi}, we prove an upper bound on the distance between a vector of divergences with respect to the multivariate Poisson process and a given centered multivariate Gaussian in the $d_2$ metric. Unlike the Wasserstein metric which remains relevant in the multivariate normal space \cite{nourdin:hal-01314406},  the $d_2$ metric is more suitable for the multi-dimensional Poisson space.\\
	

 After defining compound multivariate Hawkes process as the result of the thinning of Poisson measures in section \ref{section:thinning}, we introduce the elements of multivariate Malliavin's calculus on the Hawkes process in sections \ref{section:Malliavin} and \ref{section:Stein}. Finally, in section \ref{section:General}, we prove a general bound on a class of multivariate Hawkes functionals. As a first application in section \ref{section:CLT} we give a bound on the Wasserstein distance between the multivariate and multimarginal normalized martingale and its Gaussian limits. Then we use those results to show that if the kernels take the exponential form
 
 $$\Phi_{ij}(u)=\alpha_{ij}e^{-\beta_iu},$$
 and under the following assumptions 

 \begin{assumption}[Stability 1]\label{as:stab}
 	The spectral radius of the matrix $B^{-1}A\diag (m^1,\cdots,m^d)$ satisfies $$\rho\left ( B^{-1}A\diag (m^1,\cdots,m^d)\right ) < 1,$$
 	with $A=(\alpha_{ij})_{ij},$ $B^{-1}=\diag(\frac{1}{\beta_1},\cdots,\frac{1}{\beta_d})$ and $(m^1,\cdots,m^d)= (\E[Y^1],\cdots,\E[Y^d]).$
 	
 \end{assumption}
 \begin{assumption}[Stability 2]\label{as:vanish}
 	Set $V=B-A\diag (m^1,\cdots,m^d).$
 	The eigenvalues of $V$ are positive.
 \end{assumption}

 	\begin{assumption}[Third moment]\label{as:third} 
	The measures $\nu^1,\cdots,\nu^d$ of $Y^1_1,\cdots,Y^d_1$ have finite third moments.
\end{assumption}
\noindent 
the follwoing result holds:

\begin{theorem}
	\label{th:intro}
	Let $\boldsymbol{ L}_T$ be a multivariate compound Hawkes process.\\
	Set $\boldsymbol{ Y}'_T= \frac{\boldsymbol{ L}_T -\diag (m^1,\cdots,m^d)\left ( B -A\diag (m^1,\cdots,m^d)\right )^{-1}B\boldsymbol{\mu} T}{\sqrt T}$.\\
	Set $\tilde C = \left (J \sqrt C \right )  \left (J \sqrt C \right )^\top =JCJ^\top $, where $J= \left ( I_d-\diag(m^1,\cdots,m^d)B^{-1}A \right )^{-1}$ and   $ C=\diag(\sigma^2_1,\cdots,\sigma^2_d)$ where for any $j=1,\cdots,d$
	$$\sigma^2_j= \int x^2 \nu^j(\mathrm d x) \left[\big(B-A\diag(m^1,\cdots,m^d)\big)^{-1}B\boldsymbol{\mu}\right]^j.$$
	Let $\tilde {\boldsymbol{ G}}\sim \mathcal N \left (0, \tilde C \right )$. Then, under the Assumptions \ref{as:stab}, \ref{as:vanish}, \ref{as:third}, there exists a constant $K$ that does not depend on $T$ such that 
	$$d_2(\boldsymbol{ Y'}_T,\tilde {\boldsymbol{ G}})\leq \frac{K}{\sqrt T},$$
	for any $T>0$.
\end{theorem}
\begin{proof}
	\textit {Cf} section \ref{section:proof}.
\end{proof}
\begin{remark}
	Assumption \ref{as:stab} is a special case of the assumption $$\rho \left ( \|\Phi\|_1 \diag (m^1,\cdots,m^d) \right)<1,$$
	since in the case of an exponential kernel one has 
	$$\|\Phi\|_1= \left (\| \Phi_{ij} \|_1\right )_{ij}=\left (\| \alpha_{ij}e^{\beta_i \cdot} \|_1\right )_{ij}=\left (\frac{\alpha_{ij}}{\beta_i}\right )_{ij}=B^{-1}A.$$
	In the case of a uni-variate compound Hawkes process $(d=1)$, this is equivalent to assuming 
	$$\|\Phi\|_1m=\frac{\alpha}{\beta}m <1.$$
\end{remark}
\begin{remark}
	If the memory parameters $(\beta_i)_{i=1,\cdots,d}$ are the same for every particle (\textit{i.e} $\Phi_{ij}(u)=\alpha_{ij}e^{-\beta u}$), Assumptions \ref{as:stab} and \ref{as:vanish} become equivalent.
\end{remark}
For instance, we take the bivariate compound Hawkes process $\boldsymbol{ L}_T=\left (L^1_T,L^2_T\right )$ where 
\begin{equation*}
\left\lbrace
\begin{array}{l}
L^1_t = \sum_{k=1}^{H^1_t} Y^1_k ,\\\\
L^2_t = \sum_{k=1}^{H^2_t} Y^2_k,
\end{array}
\right.
\end{equation*}
with \textit{i.i.d} claims $(Y_k^i)_{i=1,2,k\in \mathbb N}$ with common exponential distribution $\mathcal E (1)$. The intensities are assumed to follow the dynamics 
$$
\boldsymbol{\lambda}_t = \boldsymbol{ \mu}+\int_{[0,t)}Ae^{-\beta (t-s)} \d \boldsymbol{ L}_s.
$$
For stability, it is enough to choose the parameters such that $\rho(A)<\beta$. For instance, by setting 
\begin{equation*}A=
\begin{pmatrix}
\frac{1}{2} & 2\\
2 & \frac{1}{2}
\end{pmatrix},
\end{equation*}
any $\beta > \frac{5}{2}$ satisfies Assumptions \ref{as:stab} and \ref{as:vanish}. Assumption \ref{as:third} is satisfied because the exponential distribution has moments of every order.\\
Figure \ref{fig:hist} illustrates the convergence of the compensated loss (\textit{cf.} Theorem \ref{th:intro}) to its centered Gaussian limit $\mathcal N(0,\tilde C)$ (also defined in Theorem \ref{th:intro}) 
\begin{figure}[H]
	\label{fig:hist}
	\centering
	\includegraphics[width=120mm]{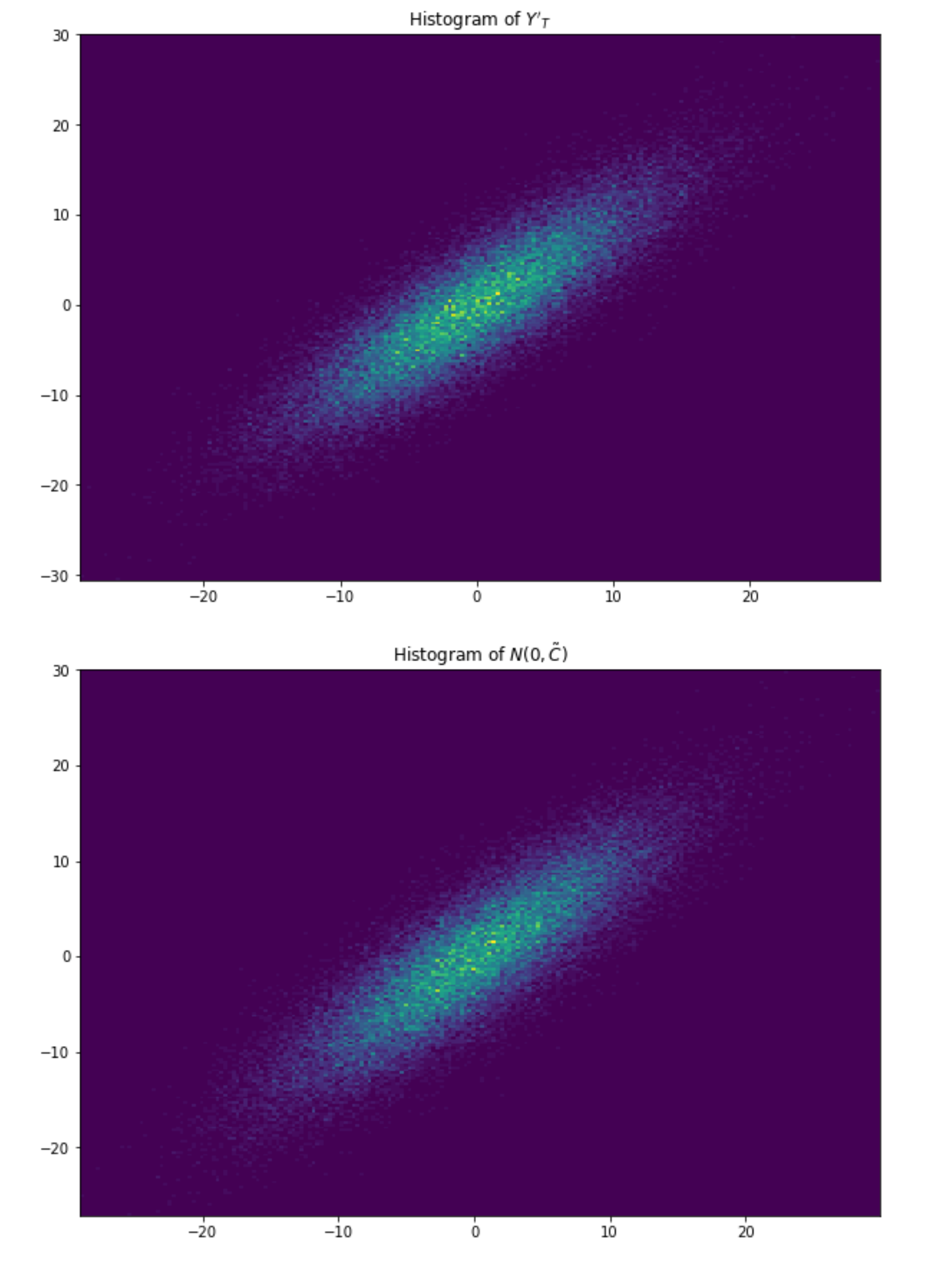}
	\caption{$2-D$ histograms of $\boldsymbol{ Y'}_T$ and $\mathcal N(0,\tilde C)$ for $\beta =4$, $\boldsymbol {\mu}=(2,3)^\top$ and $T=1000$ for $n=40000$ Monte Carlo samples.} 
\end{figure}

On the other hand, showing the speed of convergence in the $d_2$ metric is a bit more complicated, mainly for two reasons:
\begin{enumerate}
	\item We do not know for which function the upper bound is reached for the sake of simulation.
	\item Since the quantity $\mathbb E [f(\boldsymbol{ Y'}_T)]$ cannot be computed directly, is has to be approximated with a Monte Carlo estimation $\frac{1}{n}\sum_{k=1}^n f(\boldsymbol{ Y'}^k_T)$. This means that for large times, the term in $\frac{1}{\sqrt T}$ can be eclipsed by the slow decay of the Monte Carlo estimator which is in $\frac{1}{\sqrt n}$. 
\end{enumerate}	
	 We can still illustrate the behaviour of $ \frac{1}{n}\sum_{k=1}^n f(\boldsymbol{ Y'}^k_T) -\mathbb E [f(\boldsymbol{G})]$ for $\boldsymbol{G} \sim \mathcal N(0,\tilde C)$, where $f$ is a "well behaved" function for which $\mathbb E [f(\boldsymbol{ G})]$ is know explicitly. For instance, the following figure shows the evolution in time for $f(\boldsymbol{ x})=e^{-\frac{1}{4} \|\boldsymbol{ x}\|^2}$ 
	 \begin{figure}[h!]
	 	\label{fig:slope}
	 	\centering
	 	\includegraphics[width=120mm]{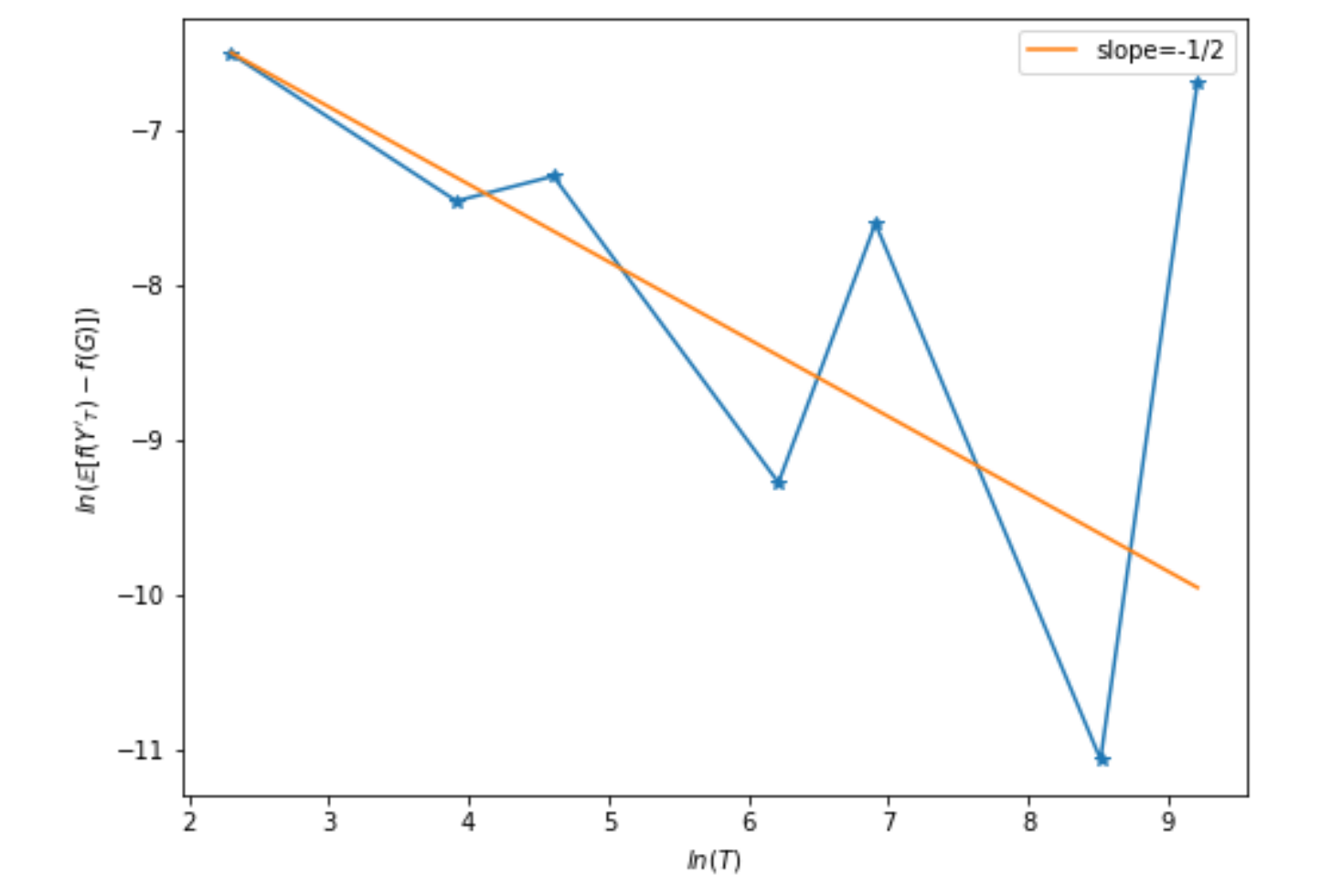}
	 	\caption{Monte Carlo estimator $\frac{1}{n}\sum_{k=1}^n f(\boldsymbol{ Y'}^k_T) -\mathbb E [f(\boldsymbol{G})] $ (in blue) for $\beta =6$, $T=[10,50,100,500,1000,5000,10000]$ and $n=150000$. Notice how for large $T$, the Monte Carlo estimator's error becomes stronger.}
	 \end{figure}
	\section{Notations and preliminaries}

	In this section we generalize the mathematical framework introduced in Section 2 of \cite{hillairet2021malliavinstein}. We then proceed to define the multivarite Hawkes loss as the result of a thinning procedure from a Poisson measure. 
	We finally recall some elements of Stein's method in 
	\subsection{General notations}
	\begin{enumerate}
		\item Let $d$ be an integer. For any two vectors $\boldsymbol{ u}=\left ( u^i\right )_{i=1,\cdots,d}$ and $\boldsymbol{ v}=\left ( v^i\right )_{i=1,\cdots,d}$ in $\R^d$, the product $\boldsymbol{ uv}$ is defined as the vector of $\R^d$ such that $\boldsymbol{ uv}=\left ( u^iv^i\right )_{i=1,\cdots,d}$
		\item Let $\theta \in \R$ and $\mathds 1_{\theta\leq u^i}$ be the indicator of the set $\{\theta \leq u^i\}$. We define $\mathbf 1_{\theta\leq \boldsymbol{ u}}$ as the element of $\R^d$ such that $\mathbf 1_{\theta \leq \boldsymbol{ u}}=\left (\mathds 1_{\theta \leq u^i} \right )_{i=1,\cdots,d}.$
		\item Similarly, if $\boldsymbol{ \nu}$ is a vector of $d$ measures and $\boldsymbol f$ is a vector of $d$ functions, we define $\int \boldsymbol f (x)\boldsymbol{ \nu}(\d x)= \left (\int f^i(x) \nu^i (\d x) \right )_{i=1,\cdots,d}.$
		\item  The inner product $\langle\cdot,\cdot\rangle$ corresponds to the Euclidian inner product and $\|\cdot\|$ is its norm.\\
		\item The operator norm of a matrix $A \in \mathcal{M}_d(\R)$ is 
		$$\|A\|_{op}:= \sup_{\|\boldsymbol x\|=1}\|A \boldsymbol x\|.$$
		\item For every function $g :\R^d \to \R$, let 
		$$\|g\|_{Lip}:= \sup _{\boldsymbol x\neq \boldsymbol y} \frac{|g(\boldsymbol x)-g(\boldsymbol y)|}{\|\boldsymbol x-\boldsymbol y\|}.$$
		\item If $g\in \mathcal C^1(\R^d)$ (continuously differentiable) then we write
		$$M_2(g):=\sup_{\boldsymbol x\neq \boldsymbol y} \frac{\|\nabla f( \boldsymbol x)-\nabla f(\boldsymbol y)\|}{\|\boldsymbol x-\boldsymbol y\|}.$$
		\item Similarly, if $g\in \mathcal C^2(\R^d)$ (twice continuousy differentiable) then 
		$$M_3(g):=\sup_{\boldsymbol x\neq \boldsymbol y} \frac{\|\Hessian f(\boldsymbol x)-\Hessian f(\boldsymbol y)\|_{op}}{\|\boldsymbol x-\boldsymbol y\|}.$$
	\end{enumerate}
	\subsection{Elements of stochastic analysis on the multivariate Poisson space}
	\label{section: Poisson}
	Let $d$ be a positive integer. Let $\nu^1,\cdots,\nu^d$ be a family of integrable probability measures on $\R_+$ such that $\nu^i(\{0\})=0$ for every $i=1,\cdots,d$, and define $m^i=\int_{\R_+} x\nu^i(\d x)$.\\
	In this section, every component of the multivariate compound  Hawkes process is obtained through the thinning of $3-$component Poisson measure.\\
	Let the space of configurations $\Omega^d$,
	where 
	
	$$ \Omega:=\left\{\omega^j=\sum_{i=1}^{n} \delta_{(t_{i},\theta_i,x_i)}, \; 0=t_0 < t_1 < \cdots < t_{n}, \; (\theta_i,x_i) \in \real_+\times \real,  \; n\in \mathbb{N}\cup\{+\infty\} \right\}.$$
	
	Let $\mathcal F$ be the $\sigma$-field associated to the vague topology on $\Omega^d$, and $\P$ the Poisson measure under which the family 
	$$\boldsymbol{N}=\big (N^j \big)_{j=1,\cdots,d}$$
	where 
	$$ N^j\left ([0,t]\times[0,b]\times(-\infty,y]\right )(\omega^j):=\omega^j \left ([0,t]\times[0,b]\times(-\infty,y] \right ), \quad  (t,b,y) \in \real_+^3, \quad  j=1,\cdots,d$$
	is a family of independent homogeneous Poisson processes with intensity measures $\d t \otimes \d {\theta} \otimes \d\nu^j$, that is, 
	
	\begin{align*}
	\P\big [N^1([0,t_1]\times[0,b_1]\times(-\infty,y_1])=n_1,\cdots,N^d([0,t_d]\times[0,b_d]\times(-\infty,y_d])=n_d \big]\\= \prod_{j=1}^d\frac{\big(t_j b_j \nu_j\big ((-\infty,y_j]\big)\big)^{n_j}\exp{\big(-t_j b_j \nu_j\big ((-\infty,y_j]\big)}\big)}{n_j!}.
	\end{align*}
	We set $\mathbb F^{\boldsymbol N}=\big(\F^{\boldsymbol N}_t\big)_{t\geq 0}$ to be the natural filtration of $\big(N^j\big)_{j=1,\cdots,d}$. The expectation with respect to $\P$ is denoted by $\E[\cdot]$. For $t\geq 0$, we denote by $\E_t[\cdot]$ the conditional expectation $\E[\cdot \vert \mathcal F^{\boldsymbol N}_t]$.\\\\
	\begin{remark}
		It is possible to define all the Poisson measures on the same probability space $\Omega^d$ by taking $$\tilde N^j\left ([0,t]\times[0,b]\times(-\infty,y]\right )(\omega)=N^j\left ([0,t]\times[0,b]\times(-\infty,y]\right )(\omega^j),\quad  (t,b,y)\in \R^3_+ ,$$
		for any $\omega=\left ( \omega^1,\cdots,\omega^d\right )\in \Omega.$\\
		From now on we confound $(\tilde N^j)_{j=1,\cdots,d}$ with $( N^j)_{j=1,\cdots,d}.$
	\end{remark}
	We now generalize the operators defined in \cite{hillairet2021malliavinstein} for the $1-$dimensional setting. We start with the component-wise shift operator. 
	\begin{definition}[Shift operator]
		\label{definition:shift}
		Let $j\in \llbracket 1,d\rrbracket $.
		We define for $(t,\theta,x)$ in $\mathbb{R}_+^3$ the measurable maps  
		$$ 
		\begin{array}{lll}
		\eps_{(t,\theta,x)}^{j+} : &\Omega &\to \Omega\\
		&(\omega^1,\cdots,\omega^j,\cdots,\omega^d) &\mapsto  \big(\omega^1,\cdots,\eps_{(t,\theta,x)}^{j+}(\omega^j),\cdots,\omega^d\big),
		\end{array}
		$$
		where for any $A$ in $\mathcal{B}(\mathbb{R}_+^2\times \real)$
		$$\left (\eps_{(t,\theta,x)}^{j+}(\omega^j)\right )(A) := \omega^j\left (A \setminus {(t,\theta,x)} \right ) + \mathds{1}_A(t,\theta,x),$$
		with 
		$$ \mathds{1}_{A}(t,\theta,x):=\left\lbrace \begin{array}{l} 1, \quad \textrm{if } (t,\theta,x)\in A,\\0, \quad \textrm{else.}\end{array}\right. $$
	\end{definition}
	
	\begin{definition}[Malliavin's derivative]
		For $F$ in $\mathcal L^2(\Omega,\mathcal F_\infty,\P)$, we define $D^j F$ the Malliavin's derivative of $F$ as 
		$$ D^j_{(t,\theta,x)} F := F\circ \eps_{(t,\theta,x)}^{j+} - F, \quad (t,\theta,x) \in \real_+^3.  $$
		If $\boldsymbol{F}=(F^1,\cdots,F^n)$ for some $n\geq 2$ where $F^i \in L^2(\Omega,\mathcal F_\infty,\P) \quad \forall i=1,\cdots,n$,
		$$ D^j_{(t,\theta,x)} \boldsymbol{F} := (F^i\circ \eps_{(t,\theta,x)}^{j+} - F^i)_{i=1,\cdots,n}, \quad (t,\theta,x) \in \real_+^3$$
	\end{definition}
	\begin{definition}
		Let $\mathcal I$  be the sub-sigma field of $\mathcal{B}(\real_+^3)\otimes \mathcal F^N$   of stochastic processes $Z:=(Z_{(t,\theta,x)})_{(t,\theta,x) \in \real_+^3}$ in $\mathcal L^1(\Omega \times \real_+^3,\P\otimes \d t\otimes \d \theta \otimes \nu)$  such that 
		$$ D^j_{(t,\theta,x)} Z_{(t,\theta,x)} = 0, \quad \textrm{ for a.a. } (t,\theta,x) \in \real^3_+, \quad \textrm{for all } j =1,\cdots,d.$$
	\end{definition}
	
	\begin{definition}[Divergence operator]
		We set $\mathcal S$ the set of stochastic processes $Z:=(Z_{(t,\theta,x)})_{(t,\theta,x) \in \real_+^3}$ in $\mathcal I$ such that for every $j \in \llbracket 1,d\rrbracket $: 
		$$ \E\left[\int_{\real_+^3} \left|Z_{(t,\theta,x)}\right|^2 \d t \d \theta \nu^j(\d x)\right] + \E\left[\left(\int_{\real_+^3} Z_{(t,\theta,x)} N^j(\d t,\d \theta,\d x)\right)^2\right]<+\infty,$$
		where $\int_{\real_+^3} Z_{(t,\theta,x)} N^j(\d t,\d \theta,\d x)$  is understood in the sense of the Stieltjes integral.\\\\ 
		For $Z$ in $\mathcal S$, we set the divergence operator with respect to $N^j$ as  
		\begin{equation}
		\label{eq:delta}
		\delta^j(Z):=\int_{\real_+^3} Z_{(t,\theta,x)} N^j(\d t,\d \theta,\d x) - \int_{\real_+^3} Z_{(t,\theta,x)} \d t \d \theta \nu^j(\d x).
		\end{equation}
		For a vector $\boldsymbol{ Z} \in \mathcal S^d$, the divergence operator with respect to $\boldsymbol{N}$ is defined as
 		$$\delta ^{\boldsymbol{N}}(\boldsymbol Z)=\sum_{i=1}^d \delta^i(Z^i),$$
		and if $Z=(\boldsymbol Z^{\cdot 1},\cdots,\boldsymbol Z^{\cdot n})\in \mathcal S^{d\times n},$ is a matrix, its divergence is defined as
		\begin{align*}
		\delta ^{\boldsymbol{N}}({Z})&=\big(\delta^{\boldsymbol{N}}(\boldsymbol Z^{\cdot 1}),\cdots,\delta^{\boldsymbol{N}}(\boldsymbol Z^{\cdot n}) \big),\\
		&=\sum_{i=1}^d \left ( \delta ^i (Z^{i1}),\cdots,\delta^i(Z^{in})\right ).
		\end{align*}
		
	\end{definition}
	
	\subsection{Definition of the multivariate compound Hawkes process}
	\label{section:thinning}
	In this subsection we give a definition of the multivariate compound Hawkes process in the Markov framework. Consider the events times  $\tau^j_1,\tau^j_2,\cdots$ associated with the $j-$th component and define the counting process 
	$$H^j_t=\sum_{i\geq 1} \mathds 1_{\tau^j_i \leq t}.$$
	Now assume that each event $\tau^j_i$ corresponds to a random "loss" $Y^j_i \sim \nu^j$ such that the variables $(Y^j_i)_{i\geq 1}$ are independent and identically distributed \textit{(i.i.d)}.\\
	The compound process $L^j_t$ of the total loss attributed to the $j-$th component is defined as
	$$L^j_t=\sum_{i\geq 1} \mathds 1_{\tau^j_i \leq t}Y^j_i.$$
	To $(H^j)_{j=1,\cdots,d}$ and $(L^j)_{j=1,\cdots,d}$ we associate a predictable intensity vector $(\lambda^j)_{j=1,\cdots,d}$ such that 
	$$\P\big [H^j_{t+\d t}-H^j_t=1 |\F_{t-}\big]=\lambda^j_t \d t$$
	which tells us how likely it is for $H^j$ to jump between $t$ and $t+\d t$, right before $t$.\\
	The process $\boldsymbol L_t=(L^1_t,\cdots,L^d_t)_{t\geq 0}$ is called a multivariate compound Hawkes process if its intensity vector $\boldsymbol{\lambda}_t$ follows the dynamics
	\begin{align*}
	\lambda^j_t&=\mu^j+\int _{[0,t)} \sum_{k=1}^d \Phi_{jk}(t-s)\d L^k_s,\\
	&=\mu^j+\sum _{k=1}^d \sum_{\tau^k_i<t} \Phi_{jk}(t-\tau^k_i) Y^k_i,
	\end{align*}
	where $(\mu^1,\cdots,\mu^d) \in \R_+^d$ and $\big(\Phi_{i,j}\big)_{i,j=1,\cdots,d}$ are non-negative integrable functions. \\\\
	
	In this article, we restrict ourselves to the case where the kernels $\Phi$ are a family of exponential functions
	$$\Phi_{jk}(u)=\alpha_{jk}e^{-\beta_j u}$$
	where $(\beta_j)_{j=1,\cdots,d} \in (\R_+)^d$ and $A=(\alpha_{ij})_{i,j=1,\cdots,d} \in \mathcal M _d(\R_+)$. The intensity can be expressed under matrix form 
	\begin{equation}
	\label{dynamics}
	\boldsymbol \lambda _t = \boldsymbol{\mu}+\int_{[0,t)} e^{-B (t-s)} A \d \boldsymbol{L}_s,
	\end{equation}
	with $B=\diag(\beta_1,\cdots,\beta_d)$.
	Furthermore, we recall Assumptions \ref{as:stab}, \ref{as:vanish}, \ref{as:third}.

	We now introduce an equivalent definition of the multivariate compound Hawkes process, presented as the result of a stochastic differential equation \textit{(SDE)} with respect to a family of random Poisson measures.
	
	\begin{theorem}
		\label{th:HRR}
		Let $\boldsymbol{N}=\big(N^1,\cdots,N^d\big)$ be a family of independent Poisson measures as presented in (Section \ref{section: Poisson}).
		Let $\boldsymbol{\mu} \in \R^d_+$  and $A\in \mathcal M _d(\R_+) $ and $\left (\beta_1,\cdots, \beta_d \right ) \in \R_+^d$ such that the Assumption \ref{as:stab} is verified. The SDE below admits a unique solution $(\boldsymbol{L},\boldsymbol{H},\boldsymbol{\lambda})$ with $\boldsymbol H$ and $\boldsymbol L$  (resp. $\boldsymbol \lambda$) $\mathbb F^ {\boldsymbol N}$-adapted (resp. $\mathbb F^{\boldsymbol N}$-predictable) 
		\begin{equation}
		\label{eq:H}
		\left\lbrace
		\begin{array}{l}
		L^j_t = \displaystyle{\int_{(0,t]\times \real_+\times \real} x \mathds{1}_{\{\theta \leq \lambda^j_s\}} N^j(\d s,\d \theta,\d x)}, \quad t \geq 0, \quad j=1,\cdots,d,\\\\
		H^j_t = \displaystyle{\int_{(0,t]\times \real_+\times \real} \mathds{1}_{\{\theta \leq \lambda^j_s\}} N^j(\d s,\d \theta,\d x)},\quad t \geq 0, \quad j=1,\cdots,d,\\\\
		\lambda^j_t = \mu^j + \int_{(0,t)}\sum_{k=1}^d \alpha _{jk} e^{-\beta_j (t-u)} \d L^k_u,\quad t \geq 0 \quad j=1,\cdots,d.
		\end{array}
		\right.
		\end{equation}
		
		We set $\mathbb F^{\boldsymbol{H}}:=(\mathcal F_t^{\boldsymbol{H}})_{t \geq 0}$ (respectively $\mathbb F^{\boldsymbol{L}}:=(\mathcal F_t^{\boldsymbol{L}})_{t \geq 0}$) the natural filtration of ${\boldsymbol{H}}$ (respectively of ${\boldsymbol{L}}$) and $\mathcal F_\infty^{\boldsymbol{H}}:=\lim_{t\to+\infty} \mathcal F_t^{\boldsymbol{H}}$ (respectively $\mathcal F_\infty^{\boldsymbol{L}}:=\lim_{t\to+\infty} \mathcal F_t^{\boldsymbol{L}}$).  Obviously $\mathcal F_t^{\boldsymbol{H}} \subset \mathcal F_t^{\boldsymbol{L}} \subset \mathcal F_t^{\boldsymbol{N}}$ as ${\boldsymbol{H}}$ is completely determined by the jump times of ${\boldsymbol{H}}$ which are exactly those of ${\boldsymbol{L}}$.\\
		
	\end{theorem}
	\begin{proof}
		This is merely a $d-$dimensional version of Theorem $2.12$ in \cite{hillairet2021malliavinstein}.
	\end{proof}
	\begin{definition}
		Let $\boldsymbol{ L}$ and $\boldsymbol{ \lambda}$ be the processes defined in Theorem \ref{th:HRR}. The multivariate Hawkes martingale $\boldsymbol{ M}$ is the process defined as $$\boldsymbol{ M}_T=\boldsymbol{ L}_T-\diag \left ( m^1,\cdots,m^d \right )\int_0^T \boldsymbol{ \lambda}_t \d t$$
		for any $T\geq 0$.
	\end{definition}
	\subsection{Malliavin's analysis of the multivariate compound Hawkes process}
	\label{section:Malliavin}
	This subsection is a generalization of Section 2.3 of \cite{hillairet2021malliavinstein} where we examine the impact of the $i-$th Malliavin's derivative on the $j-$th component of $\boldsymbol{L}$, $\boldsymbol{H}$ and $\boldsymbol{\lambda}$.
	
	\begin{defprop}[Lemma 2.15 in \cite{hillairet2021malliavinstein}]
		\label{lemma:TempDH}
		Let $t$ and $v$ in $\real_+$ and $ (\theta,\theta_0,x)$ in $\real_+^3$. For every $j\in \llbracket 1,d\rrbracket$ it holds that :
		\begin{align*}
		& \quad \mathds{1}_{\{\theta \leq \lambda^j_t\}} (\boldsymbol{L}_v\circ \eps_{(t,\theta,x)}^{j+},\boldsymbol{\lambda}_v\circ \eps_{(t,\theta,x)}^{j+})_{v\geq 0} \\
		&= \mathds{1}_{\{\theta_0 \leq \lambda^j_t\}} (\boldsymbol{L}_v\circ \eps_{(t,\theta_0,x)}^{j+},\boldsymbol{\lambda}_v\circ \eps_{(t,\theta_0,x)}^{j+})_{v\geq 0}.
		\end{align*}
		which entails that for every -eventually vectorial- $\mathcal F_\infty^{\boldsymbol{L}}$-measurable random variable $\boldsymbol{F}$
		$$ \mathds{1}_{\{\theta \leq \lambda^j_t\}} D^j_{(t,\theta,x)} \boldsymbol{F} = \mathds{1}_{\{\theta_0 \leq \lambda^j_t\}} D^j_{(t,\theta_0,x)} \boldsymbol{F}, \quad \P-a.s.$$
		where the derivative is applied to each component of $\boldsymbol{F}.$
		The last equation allows us to define 
		$$ D^j_{(t,\lambda^j_t,x)} \boldsymbol{F}:= \mathds{1}_{\{\theta \leq \lambda^j_t\}} D^j_{(t,\theta,x)} \boldsymbol{F}, \quad \forall (\theta,x) \in \real_+^2$$
		as well as its vector version 
		\begin{align*}
		D_{(t,\boldsymbol \lambda_t,x)} \boldsymbol{F}&= \big( D^j_{(t,\lambda^j_t,x)} \boldsymbol{F}\big)_{j=1,\cdots,d},\\
		&:= \big(\mathds{1}_{\{\theta \leq \lambda^j_t\}} D^j_{(t,\theta,x)} \boldsymbol{F}\big)_{j=1,\cdots,d}, \quad \forall (\theta,x) \in \real_+^2.
		\end{align*}
	\end{defprop}
	\begin{proof}
		For each $i \in \llbracket 1,d\rrbracket$ and $v\geq t$, we have conditionally on $\theta_0 \leq \lambda^j_t$
		\begin{align*}
		L^i_v \circ \eps_{(t,\theta_0,x)}^{+j} 
		&=L^i_{t-} + x \mathds{1}_{\{\theta_0 \leq \lambda^j_t\}}\mathds 1_{i=j} + \int_{(t,v] \times \real} \int_{\real_+} y \mathds{1}_{\{\theta \leq \lambda^i_u \circ \eps_{(t,\theta_0,x)}^{+j}\}} N^i(\d \theta,\d u,\d y),
		\end{align*}
		and $L^i_v \circ \eps_{(t,\theta_0,x)}^{+j} =L^i_{v}$ if $t>v$.
		Similarly for $\boldsymbol{\lambda}$, we have 
		\begin{align*}
		\lambda^i_v \circ \eps_{(t,\theta_0,x)}^{+j} 
		&=\left(\mu_i + \int_{(0,t)}\sum_{k=1}^d e^{-\beta(v-u)}\alpha _{ik} \d L^k_u + \int_{[t,v)} \sum_{k=1}^d e^{-\beta(v-u)}\alpha _{ik} \d L^k_u \right) \circ \eps_{(t,\theta_0,x)}^{+j} \\
		&=\mu_i + \int_{(0,t)} \sum_{k=1}^d e^{-\beta(v-u)}\alpha _{ik} \d L^k_u  +x \alpha_{ij}e^{-\beta(v-t)} + \int_{(t,v)} \sum_{k=1}^d e^{-\beta(v-u)}\alpha _{ik}  \d (L^k_u\circ \eps_{(t,\theta_0,x)}^{+j}).
		\end{align*}
		In other words, $(\boldsymbol{L} \circ \eps_{(t,\theta_0,x)}^{+j},\boldsymbol{\lambda} \circ \eps_{(t,\theta_0,x)}^{+j})$ solves the same (path-wise and in the SDE sense) equation for any $\theta_0$ such that $\theta_0 \leq \lambda_t$.\\
		For the second equality, \textit{cf.} the proof of Proposition 2.16 in \cite{hillairet2021malliavinstein}.
	\end{proof}
	\begin{remark}
		It is also possible to include the process $\boldsymbol{H}$ in these results. From now on, $\boldsymbol{H}$ will be omitted and we will focus exclusively on $\boldsymbol{L}$ and $\boldsymbol{\lambda}$.
	\end{remark}
	Now we give the Malliavin's derivative of the multivariate compound Hawkes process as well as its intensity. To do so we start with introducing some notations. The vector $\boldsymbol{e}_i$ is the element of $\R^d$ that has $1$ in the $i-$th component and zero elsewhere.
	\begin{prop}
		\label{prop:DescDH}
		Let $t\geq 0$ and $x\in \real$. For every $j \in \llbracket 1,d\rrbracket  $, we have
		\begin{equation*}
		(D^j_{(t,\lambda^j_t,x)} \boldsymbol{L}_s,D^j_{(t,\lambda^j_t,x)}\boldsymbol{\lambda}_s )  = \left\lbrace
		\begin{array}{l}
		( x \boldsymbol{e}_j+\hat {\boldsymbol{L}}_s^{j,t,x},\hat {\boldsymbol{\lambda}}_s^{j,t,x}),   \quad   s\geq t, \\\\
		(0,0), \quad \quad \quad \quad  s<t
		\end{array}
		\right.
		\end{equation*}
		where the equality is understood path-wise and in the SDE sense and 
		where\\ $(\hat {\boldsymbol{L}}_s^{j,t,x},\hat {\boldsymbol{\lambda}}_s^{j,t,x})_{s\geq t}=\big ( (\hat {{L}}_{s}^{i,j,t,x},\hat {{\lambda}}_{s}^{i,j,t,x})_{s\geq t}\big)_{i=1,\cdots,d}$  is the unique solution to the SDE 
		\begin{equation}
		\label{eq:DH}
		\left\lbrace
		\begin{array}{l}
		\hat {{L}}_{s}^{i,j,t,x} = \displaystyle{\int_{(t,s]\times \real_+^2} y \mathds{1}_{\{\lambda^i_u \leq \theta \leq \lambda^i_u+ \hat \lambda_u^{i,j,t,x}\}} N^i(\d u,\d \theta,\d y)}, \quad s \geq t,\\\\
		\hat \lambda_s^{i,j,t,x} = x \alpha_{ij} e^{-\beta_i(s-t)} + \displaystyle{\sum_{k=1}^d} \displaystyle{\int_{(t,s)}\alpha_{ik} e^{-\beta_i(s-u)} \d\hat L_u^{k,j,t}}, \quad s>t, \; \hat \lambda_t^{i,j,t,x}=x\alpha_{ij},
		\end{array}
		\right.
		\end{equation}
		or under the matrix form 
		\begin{equation}
		\left\lbrace
		\begin{array}{l}
		\hat {\boldsymbol{L}}_{s}^{j,t,x} = \displaystyle{\int_{(t,s]\times \real_+^2} y \mathbf{1}_{\{\boldsymbol{\lambda}_u \leq \theta \leq \boldsymbol{\lambda}_u+ \hat{\boldsymbol{\lambda}}_u^{j,t,x}\}} \boldsymbol N(\d u,\d \theta,\d y)}, \quad s \geq t,\\\\
		\hat {\boldsymbol \lambda}_s^{j,t,x} = x e^{-B(s-t)}\boldsymbol A_{\cdot j}  +  \displaystyle{\int_{(t,s)} e^{-B(s-u)} A \d\hat {\boldsymbol L}_u^{j,t,x}}, \quad s>t, \; \hat {\boldsymbol \lambda}_t^{j,t,x}=x\boldsymbol A_{\cdot j},
		\end{array}
		\right.
		\end{equation}
		where the first integral is understood as a vector of integrals with respect to each Poisson measure and the indicator of a vector is understood as a vector of indicators.\\
		The process $(\hat {\boldsymbol L }_s^{j,t,x},\hat {\boldsymbol \lambda}_s^{j,t,x})_{ s \in [t,+\infty)}$ is a generalized multivariate compound Hawkes process, with an intensity vector that is not bounded away from $0$.
	\end{prop}
	\begin{proof}
		It is sufficient to apply the same procedure of Proposition 2.19 's proof  in \cite{hillairet2021malliavinstein} to each component of the vectors. Note that in this proposition, the subscript contains $x$, because in our case the size of the jump has an impact on the $\hat {\boldsymbol{\lambda}}$'s behaviour, since it is an integral with respect to $\d\hat {\boldsymbol{L}}$ and not merely $\d\hat {\boldsymbol{H}}$. 
	\end{proof}
	\begin{remark}
		\label{remark:tilde}
		Note that the process $\left (\hat {\boldsymbol{L}}_s^{j,t,x},\hat {\boldsymbol{\lambda}}_s^{j,t,x} \right )_{s\geq t}$ defined above is equal in distribution to the (generalized) Hawkes process $\left (\tilde {\boldsymbol{L}}_s^{j,t,x},\tilde {\boldsymbol{\lambda}}_s^{j,t,x} \right )_{s\geq t}$ defined as a solution to the SDE 
		\begin{equation}
		\left\lbrace
		\begin{array}{l}
		\tilde {\boldsymbol{L}}_{s}^{j,t,x} = \displaystyle{\int_{(t,s]\times \real_+^2} y \textbf{1}_{\{\boldsymbol 0 \leq \theta \leq  \tilde{\boldsymbol{\lambda}}_u^{j,t,x}\}} \tilde {\boldsymbol N}(\d u,\d \theta,\d y)}, \quad s \geq t,\\\\
		\tilde {\boldsymbol \lambda}_s^{j,t,x} = x e^{-B(s-t)}\boldsymbol A_{\cdot j}  +  \displaystyle{\int_{(t,s)} e^{-B(s-u)} A \d\tilde {\boldsymbol L}_u^{j,t,x}}, \quad s>t, \; \tilde {\boldsymbol \lambda}_t^{j,t,x}=x\boldsymbol A_{\cdot j},
		\end{array}
		\right.
		\end{equation}
		where $\tilde {\boldsymbol{ N}}$ is a family of Poisson measures independent from $\boldsymbol{ N}$ but have the same distributions.
	\end{remark}
	We conclude this section by stating the multivariate integration by parts (IBP) formula for the compound Hawkes martingale.

	\subsection{Multivariate Stein's method}
	\label{section:Stein}
	Stein's method is based on an alternative characterization of the Gaussian distribution. The combination of this characterization with elements of Malliavin's calculus (known as the Nourdin-Peccati approach) provides us with a way to estimate the distance between a random variable and a Gaussian.
	\begin{lemma}
		Let $C$ be a $d\times d$ real symmetric positive definite matrix and $\boldsymbol Y$ an $\R^d$ random variable. Then $\boldsymbol Y \sim \mathcal N (0,C)$ if and only if for every twice differentiable function $f: \R^d \to \R$ such that  $\E\big[|\langle C, \Hessian f(\boldsymbol Y)\rangle _{H.S}|+| \langle \boldsymbol Y,\nabla f(\boldsymbol Y)\rangle|\big] < +\infty$:
		$$\E\big[\langle C, \Hessian f(\boldsymbol Y)\rangle _{H.S}- \langle \boldsymbol Y,\nabla f(\boldsymbol Y)\rangle\big]=0$$
		where $\langle A,B \rangle_{H.S}= tr(AB^T)=\sum A_{ij}B_{ij}$ is the Hilbert-Schmidt matrix inner product and $\Hessian$ is the Hessian operator.
	\end{lemma}
	\begin{proof}
		\textit{Cf.} \cite{hillairet2021malliavinstein} for example.
	\end{proof}
	This shows that if for a sufficiently large class of functions $f$ one has 
	$$\E\big[\langle C, \Hessian f(\boldsymbol Y)\rangle _{H.S}- \langle \boldsymbol Y,\nabla f(\boldsymbol Y)\rangle\big]\simeq 0$$
	in some sense, then the random variable $\boldsymbol Y$'s distribution is fairly "close" to $\mathcal{N}(0,C)$.\\
	We now introduce the function classes as well as the metric that we will use to quantify how close the given distance of a random variable to a Gaussian distribution.
	\begin{definition}
		\begin{enumerate}
			\item The distance $d_2$ between two integrable random variables $\boldsymbol X$ and $\boldsymbol Y$ is given by 
			$$d_2(\boldsymbol X,\boldsymbol Y):= \sup_{f\in \mathcal H} |\E[f(\boldsymbol X)]-\E[f(\boldsymbol Y)]|,$$
			where
			$$\mathcal H :=\{g \in \mathcal C^2(\R^d), ~\text{such that } \|g\|_{Lip} \leq 1 \text{ and } M_2(g)\leq 1 \}.$$
			\item For a fixed matrix $C\in \mathcal S^{++}_d(\R)$ we define $\F_C$ to be the functional space 
			$$\F_C:= \{g \in \mathcal C^2(\R^d), ~\text{such that } M_2(g)\leq \|C^{-1}\|_{op} \|C\|_{op}^{1/2}~ \text{, }M_3(g)\leq \frac{\sqrt{2\pi}}{4}\|C^{-1}\|_{op}^{3/2} \|C\|_{op} \}.$$

		\end{enumerate}
		
	\end{definition} 
	Combining the definitions with the multivariate Stein's equation, it is possible to prove (following the lines of Lemma 2.17 in \cite{giovanni2010multi}) the following estimate of the $d_2$ distance between a centered variable $\boldsymbol F$ and $\boldsymbol G \sim \mathcal N_d(0,C)$:
	$$d_2(\boldsymbol F,\boldsymbol G)\leq \sup _{f\in \mathcal F_C} \big |\E\big[\langle C, \Hessian f(\boldsymbol F)\rangle _{H.S}- \langle \boldsymbol F,\nabla f(\boldsymbol F)\rangle\big]\big|.$$
	\begin{remark}
		Unlike for the univariate Poisson space case (\textit{cf.} \cite{hillairet2021malliavinstein}) or the multivariate variables on a Gaussian space \cite{nourdin2010multivariate}, the Wasserstein distance $d_W$ is not well adapted to our computations. Nevertheless, if $d_2(\boldsymbol F_T,\boldsymbol G)\longrightarrow 0$ as $T$ goes to infinity implies that $\boldsymbol F_T$ converges to the Gaussian $\boldsymbol G$ in distribution (\textit{cf.} Remark 2.16 in \cite{giovanni2010multi}) .
	\end{remark}
	\section{Main results}
	\label{section:main}
\subsection{General bound}
\label{section:General}
In this section, we give a bound on the $d_2$ distance between any random variable that can be expressed as a divergence with respect to the Poisson measure.
	\begin{lemma}
		\label{lemma:taylor}
		Let n be an integer and let $\boldsymbol{F}=(F^1,\cdots,F^n)$ be an $\F^{\boldsymbol L}_{\infty}$ such that  $\E [\|\boldsymbol F\|^2]\leq +\infty$. For all $\phi \in \mathcal C^3 (\R^n )$ with bounded derivatives, for any $k=1,\cdots,n$, there exists a random $\bar {\boldsymbol{F}}$ such that 
		$$D^k_{(t,\theta,x)}\partial_i\phi(\boldsymbol{F})=\langle \nabla \partial_i\phi(\boldsymbol F),D^k_{(t,\theta,x)} \boldsymbol F\rangle + \frac{1}{2}\langle D^k_{(t,\theta,x)}\boldsymbol{F} , \Hessian \partial_i\phi(\bar{\boldsymbol{F}})D^k_{(t,\theta,x)}\boldsymbol{F} \rangle,  ~ (t,\theta,x)\in \R_+^3$$
		where 
		$$\left |\langle \boldsymbol y , \Hessian \partial_k\phi (\bar {\boldsymbol{F}}) \boldsymbol{y}\rangle \right | \leq  \|\boldsymbol y\|^2 M_3(\phi), ~ \boldsymbol{y} \in \R^d$$
	.
		
	\end{lemma}
	\begin{proof}
		The equality is merely an application of the multivariate Taylor-Young Theorem, combined with the fact that
		\begin{align*}
		D^k_{(t,\theta,x)}\partial_i\phi(\boldsymbol{F})&=\partial_i\phi(\boldsymbol{F}\circ \varepsilon^{k+}_{(t,\theta,x)})-\partial_i\phi(\boldsymbol{F}),\\
		&=\partial_i\phi(\boldsymbol{F}+D^k_{(t,\theta,x)} \boldsymbol F)-\partial_i\phi(\boldsymbol{F}).
		\end{align*} 
		When it comes to the upper bound on the rest, we have using Cauchy-Schwarz
		\begin{align*}
		\left |\langle \boldsymbol y , \Hessian \partial_i\phi (\bar {\boldsymbol{F}}) \boldsymbol{y}\rangle \right |&\leq \|\boldsymbol{y}\| \|\Hessian \partial_i\phi (\bar {\boldsymbol{F}}) \boldsymbol{y}\|,\\
		&\leq \|\boldsymbol y\|^2 \|\Hessian \partial_i\phi (\bar {\boldsymbol{F}})\|_{op}.
		\end{align*}
		Using Schwarz's Theorem we have that 
		\begin{align*}
		\Hessian \partial_i\phi (\bar {\boldsymbol{F}})&=\partial_i \Hessian\phi (\bar {\boldsymbol{F}}),\\
		&=\lim_{h\to 0}\frac{\Hessian \phi (\bar {\boldsymbol{F}}+h\boldsymbol{ e}_i)-\Hessian \phi (\bar {\boldsymbol{F}})}{h},
		\end{align*} 
		and the result follows using the fact that the norm $\|\cdot\|_{op}$ is continuous.
	\end{proof}
This lemma will be useful in proving the following result.
	
	\begin{theorem}
		\label{th:Anthony}
		Let $n \in \mathbb N^*$. Let $Z=\left ( Z^{ki}_{(t,\theta,x)}\right )_{(t,\theta,x)\in \R^3_+}$ be a stochastic process in $\mathcal S ^{d\times n}$. Set $\boldsymbol{ F}=\delta^{\boldsymbol{ N}} (Z)$ the divergence of $Z$. Then, letting $\boldsymbol{ G} \sim \mathcal N (0,C)$ (for $C\in \mathcal S_n^{++}(\R)$) 
		we have
		\begin{align*}
		d_2(\boldsymbol{ F},\boldsymbol{ G})
		\leq& \|C^{-1}\|_{op} \|C\|_{op}^{1/2}\sum_{i,j=1}^n \E \left [\left |C_{ij}-\sum_{k=1}^d\int _{\R^3_+}Z^{ki}_{(t,\theta,x)}D^k_{(t,\theta,x)}  F^j \mathrm d t \mathrm \d \theta \nu^k(\mathrm \d x) \right |\right ]\\
		&+\frac{\sqrt{2\pi}}{8}\|C^{-1}\|_{op}^{3/2} \|C\|_{op}  \sum_{i,j=1}^n \sum_{k=1}^d \E \left[ \int _{\R^3_+}\left |Z^{ki}_{(t,\theta,x)}\right | \left \| D^k_{(t,\theta,x)} \boldsymbol{ F}\right \|^2\mathrm d t \mathrm \d \theta \nu^k(\mathrm \d x)\right].\\
		\end{align*}
	 If in particular  $Z^{ki}_{(t,\theta,x)}=\mathds 1_{\theta \leq \lambda^k_t}U^{ki}_{(t,x)}$, the upper bound takes the form 
	 \begin{align*}
	 d_2(\boldsymbol{ F},\boldsymbol{ G})\leq & \|C^{-1}\|_{op} \|C\|_{op}^{1/2}\sum_{i,j=1}^n \E \left [\left |C_{ij}-\sum_{k=1}^d\int _{\R^2_+}\lambda^k_tU^{ki}_{(t,x)}D^k_{(t,\lambda^k_t,x)}  F^j \mathrm d t \nu^k(\mathrm \d x) \right |\right ]\\
	 &+\frac{\sqrt{2\pi}}{8}\|C^{-1}\|_{op}^{3/2} \|C\|_{op}  \sum_{i,j=1}^n \sum_{k=1}^d \E \left[ \int _{\R^2_+}\lambda^k_t \left |U^{ki}_{(t,x)}\right | \left \| D^k_{(t,\lambda^k_t,x)}  \boldsymbol{ F}\right \|^2\mathrm d t  \nu^k(\mathrm \d x)\right].
	 \end{align*}
	\end{theorem}
\begin{proof}
We start by recalling the bound on the $d_2$ distance between $\boldsymbol F$ and $\boldsymbol{G}$
$$d_2(\boldsymbol{F},\boldsymbol{G})\leq \sup _{f\in \mathcal F_C} \big |\E\big[\langle C, \Hessian f(\boldsymbol F_T)\rangle _{H.S}- \langle \boldsymbol F,\nabla f(\boldsymbol F)\rangle\big]\big|.$$ 
In \cite{nourdin2010multivariate}, the following technique is used. For any $f\in \mathcal F_C$ 	and $\eta >0$ set $f_\eta(x)=\E[f(x+\sqrt \eta N)]$ where $N$ is a centered Gaussian of unit variance. It is easy to see that 
\begin{enumerate}
	\item $f_\eta \in \mathcal C^{\infty} (\R^n).$
	\item $\|f-f_\eta\|_{\infty}\longrightarrow 0$ when $\eta$ goes to zero.
	\item $M_2(f_\eta )\leq M_2(f)$ and $M_3(f_\eta)\leq M_3(f)$ using Young's convolution inequality.
\end{enumerate}
Thus we can assume that $f\in \mathcal C^{\infty}(\R^n)$ and the computations yield
\begin{align*}
\E\big[\langle C, \Hessian f(\boldsymbol F)\rangle _{H.S}- \langle \boldsymbol F,\nabla f(\boldsymbol F)\rangle\big] =& \E \big [\sum_{i,j=1}^n C_{ij} \partial^2_{ij}f(\boldsymbol{F}) -\sum_{i=1}^n F^i\partial_i f(\boldsymbol{F})\big],\\
=&\sum_{i,j=1}^n C_{ij} \E \big [\partial^2_{ij}f(\boldsymbol{F})\big] -\sum_{i=1}^n \E \big [F^i\partial_i f(\boldsymbol{F})\big].\\
\end{align*}
By definition, $\boldsymbol{ F}$ is defined as the divergence of the matrix $Z$, thus for each $i\in \llbracket 1,n\rrbracket $ 
$$F^i=\sum _{k=1}^d\delta^k (Z^{ki}),$$
which entails that
\begin{align*}
\E \big [F^i\partial_i f(\boldsymbol{F})\big]&=\sum_{k=1}^d \E \big [\delta^k (Z^{ki})\partial_i f(\boldsymbol{F})\big].
\end{align*}
Set $\E_{\neq k}[\cdot]=\E [\cdot | N^1,\cdots,N^{k-1},N^{k+1},\cdots,N^d]$, which stands for the expected value knowing all the counting measures except for the $k-$th one. This notation is introduced in order use the integration by parts formula in \cite{picard1996formules} which is available for univariate processes. Hence 
\begin{align*}
\E \big [F^i\partial_i f(\boldsymbol{F})\big]&= \sum_{k=1}^d \E \big [ \E_{\neq k}[\delta^k (Z^{ki})\partial_i f(\boldsymbol{F})]\big],\\
&=\sum_{k=1}^d \E \big [ \E_{\neq k}\left[\int \int\int_{\R^3_+}Z^{ki}_{(t,\theta,x)}D^k_{(t,\theta,x)}\partial_i f(\boldsymbol{F})\mathrm d t \mathrm \d \theta \nu^k(\mathrm \d x)\right ]\big],\\
&=\sum_{k=1}^d \E \left[ \int _{\R^3_+}Z^{ki}_{(t,\theta,x)}D^k_{(t,\theta,x)}\partial_i f(\boldsymbol{F})\mathrm d t \mathrm \d \theta \nu^k(\mathrm \d x)\right].\\
\end{align*}
Using Lemma \ref{lemma:taylor}, we have that 
\begin{align*}
D^k_{(t,\theta,x)}\partial_i f(\boldsymbol{F})&=\langle \nabla \partial_i f(\boldsymbol F),D^k_{(t,\theta,x)} \boldsymbol F\rangle + \frac{1}{2}\langle D^k_{(t,\theta,x)}\boldsymbol{F} , \Hessian \partial_i f(\bar{\boldsymbol{F}})D^k_{(t,\theta,x)}\boldsymbol{F} \rangle,\\
&=\sum_{j=1}^n\partial^2_{ij} f(\boldsymbol F)D^k_{(t,\theta,x)}  F^j + \frac{1}{2}\langle D^k_{(t,\theta,x)}\boldsymbol{F} , \Hessian \partial_i f(\bar{\boldsymbol{F}})D^k_{(t,\theta,x)}\boldsymbol{F} \rangle,\\
\end{align*}
for some random $\bar{\boldsymbol{F}}$. Hence
\begin{align*}
\E \big [F^i\partial_i f(\boldsymbol{F})\big]=&\sum_{k=1}^d \E \left [\int _{\R^3_+}Z^{ki}_{(t,\theta,x)} \sum_{j=1}^n\partial^2_{ij} f(\boldsymbol F)D^k_{(t,\theta,x)}  F^j \mathrm d t \mathrm \d \theta \nu^k(\mathrm \d x) \right ]\\ &+\sum_{k=1}^d\frac{1}{2} \E \left[ \int _{\R^3_+}Z^{ki}_{(t,\theta,x)} \langle D^k_{(t,\theta,x)}\boldsymbol{F} , \Hessian \partial_i f(\bar{\boldsymbol{F}})D^k_{(t,\theta,x)}\boldsymbol{F} \rangle\mathrm d t \mathrm \d \theta \nu^k(\mathrm \d x)\right],\\
=&\sum_{j=1}^n \E \left [\partial^2_{ij} f(\boldsymbol F)\sum_{k=1}^d\int _{\R^3_+}Z^{ki}_{(t,\theta,x)}D^k_{(t,\theta,x)}  F^j \mathrm d t \mathrm \d \theta \nu^k(\mathrm \d x) \right ]\\ &+\sum_{k=1}^d\frac{1}{2} \E \left[ \int _{\R^3_+}Z^{ki}_{(t,\theta,x)} \langle D^k_{(t,\theta,x)}\boldsymbol{F} , \Hessian \partial_i f(\bar{\boldsymbol{F}})D^k_{(t,\theta,x)}\boldsymbol{F} \rangle\mathrm d t \mathrm \d \theta \nu^k(\mathrm \d x)\right].
\end{align*}
Then
\begin{align*}
\E\big[\langle C, \Hessian f(\boldsymbol F)\rangle _{H.S}- \langle \boldsymbol F,\nabla f(\boldsymbol F)\rangle\big] =& \sum_{i,j=1}^n \E \left [\partial^2_{ij} f(\boldsymbol F)\left (C_{ij}-\sum_{k=1}^d\int _{\R^3_+}Z^{ki}_{(t,\theta,x)}D^k_{(t,\theta,x)}  F^j \mathrm d t \mathrm \d \theta \nu^k(\mathrm \d x) \right )\right ]\\
&-\frac{1}{2}\sum_{i=1}^n \sum_{k=1}^d \E \left[ \int _{\R^3_+}Z^{ki}_{(t,\theta,x)} \langle D^k_{(t,\theta,x)}\boldsymbol{F} , \Hessian \partial_i f(\bar{\boldsymbol{F}})D^k_{(t,\theta,x)}\boldsymbol{F} \rangle\mathrm d t \mathrm \d \theta \nu^k(\mathrm \d x)\right].
\end{align*}
By taking the absolute value, using the triangular inequality and the fact that for any $f\in \mathcal F_C$ and any $(i,j) \in \llbracket 1,d\rrbracket ^2$ $$\big \| \partial^2 _{ij} f\big \|_{\infty}\leq M_2(f)\leq \|C^{-1}\|_{op} \|C\|_{op}^{1/2}$$
we obtain the bound 
\begin{align*}
|\E\big[\langle C, \Hessian f(\boldsymbol F)\rangle _{H.S}- \langle \boldsymbol F,\nabla f(\boldsymbol F)\rangle\big]| \leq& \|C^{-1}\|_{op} \|C\|_{op}^{1/2}\sum_{i,j=1}^n \E \left [\left |C_{ij}-\sum_{k=1}^d\int _{\R^3_+}Z^{ki}_{(t,\theta,x)}D^k_{(t,\theta,x)}  F^j \mathrm d t \mathrm \d \theta \nu^k(\mathrm \d x) \right |\right ]\\
&+\frac{1}{2} \sum_{i,j=1}^n \sum_{k=1}^d \E \left[ \int _{\R^3_+}|Z^{ki}_{(t,\theta,x)}| \left |\langle D^k_{(t,\theta,x)}\boldsymbol{F} , \Hessian \partial_i f(\bar{\boldsymbol{F}})D^k_{(t,\theta,x)}\boldsymbol{F} \rangle\right |\mathrm d t \mathrm \d \theta \nu^k(\mathrm \d x)\right],\\
\leq& \|C^{-1}\|_{op} \|C\|_{op}^{1/2}\sum_{i=1}^n \E \left [\left |C_{ij}-\sum_{k=1}^d\int _{\R^3_+}Z^{ki}_{(t,\theta,x)}D^k_{(t,\theta,x)}  F^j \mathrm d t \mathrm \d \theta \nu^k(\mathrm \d x) \right |\right ]\\
&+\frac{\sqrt{2\pi}}{8}\|C^{-1}\|_{op}^{3/2} \|C\|_{op}  \sum_{i=1}^n \sum_{k=1}^d \E \left[ \int _{\R^3_+}|Z^{ki}_{(t,\theta,x)}| \left \| D^k_{(t,\theta,x)}  \boldsymbol{ F}\right \|^2\mathrm d t \mathrm \d \theta \nu^k(\mathrm \d x)\right],\\
\end{align*}
and finally 
\begin{align*}
d_2(\boldsymbol{ F},\boldsymbol{ G}) \leq & \|C^{-1}\|_{op} \|C\|_{op}^{1/2}\sum_{i,j=1}^n \E \left [\left |C_{ij}-\sum_{k=1}^d\int _{\R^3_+}Z^{ki}_{(t,\theta,x)}D^k_{(t,\theta,x)}  F^j \mathrm d t \mathrm \d \theta \nu^k(\mathrm \d x) \right |\right ]\\
&+\frac{\sqrt{2\pi}}{8}\|C^{-1}\|_{op}^{3/2} \|C\|_{op}  \sum_{i=1}^n \sum_{k=1}^d \E \left[ \int _{\R^3_+}|Z^{ki}_{(t,\theta,x)}| \left \| D^k_{(t,\theta,x)}  \boldsymbol{ F}\right \|^2\mathrm d t \mathrm \d \theta \nu^k(\mathrm \d x)\right].
\end{align*}
For the second point, assume that $\forall k \in \llbracket 1,d \rrbracket$ and $\forall i\in \llbracket 1,n \rrbracket$ 
$$Z^{ki}_{(t,\theta,x)}=\mathds 1_{\theta \leq \lambda^k_t}U^{ki}_{(t,x)}.$$
Keeping in mind that according to Definition \ref{lemma:TempDH}, we have
$$\mathds 1_{\theta \leq \lambda^k_t}D^k_{(t,\theta,x)}  F^j=D^k_{(t,\lambda^k_t,x)} F^j$$
whenever $\theta$ is less than $\lambda^k_t$, thus
\begin{align*}
d_2(\boldsymbol{ F},\boldsymbol{ G})\leq & \|C^{-1}\|_{op} \|C\|_{op}^{1/2}\sum_{i,j=1}^n \E \left [\left |C_{ij}-\sum_{k=1}^d\int _{\R^2_+}\lambda^k_tU^{ki}_{(t,x)}D^k_{(t,\lambda^k_t,x)}  F^j \mathrm d t \nu^k(\mathrm \d x) \right |\right ]\\
&+\frac{\sqrt{2\pi}}{8}\|C^{-1}\|_{op}^{3/2} \|C\|_{op}  \sum_{i=1}^n \sum_{k=1}^d \E \left[ \int _{\R^2_+}\lambda^k_t|U^{ki}_{(t,x)}| \left \| D^k_{(t,\lambda^k_t,x)}  \boldsymbol{ F}\right \|^2\mathrm d t  \nu^k(\mathrm \d x)\right].
\end{align*}
\end{proof}
\subsection{Bounds on the CLTs}
\label{section:CLT}
	The following theorem is our first main result where we give the bound on the $d_2$ distance between a vector formed by the normalized Hawkes martingale 
	$$\boldsymbol{F}_T=\frac{\boldsymbol{M}_T}{\sqrt T}=\frac{\boldsymbol{L}_T-\diag (m^1,\cdots,m^d)\int_0^T \boldsymbol{\lambda}_t \d t}{\sqrt T}$$
	and its Gaussian limit as $T$ goes to infinity.
	
	\begin{theorem}
		\label{th:multi-marginal}
		Fix $p$ and $d$ in $\mathbb N^*$. Let $(\boldsymbol{L}_t)_{t \geq 0}$ be a compound multivariate Hawkes process whose intensity $\boldsymbol \lambda$ follows the dynamics \ref{dynamics} and let $0<v_1<\cdots<v_p\leq 1$  be $p$ distinct positive numbers. Set $$\boldsymbol{ \Gamma}_T=\left (F^1_{v_1T},\cdots,F^1_{v_pT},\cdots, F^d_{v_1T},\cdots ,F^d_{v_pT} \right ) \in \R^{p\cdot d},$$
		and 
		$$\hat C = \diag (C^1,\cdots,C^d)\in \mathcal S_{p\cdot d}(\R)$$
		 the block diagonal matrix such that $\forall n \in \llbracket 1,d \rrbracket$
		 $$C^n_{ij}=C^n_{ji}=\int x^2 \nu^n(\d x) \sqrt {\frac{v_i}{v_j}}\left [ \left (B-A\diag (m^1,\cdots,m^d) \right )^{-1}B \boldsymbol{ \mu}\right ]^n, \quad \forall 1\leq i \leq j \leq p.$$
		 Let $\boldsymbol{ G} \sim \mathcal N (0,\hat C)$. Then there is a constant $K>0$ independent from $T$ such that 
		 $$d_2(\boldsymbol{ \Gamma}_T,\boldsymbol{ G})\leq \frac{K}{\sqrt T}.$$
	\end{theorem}
\begin{proof}
			Throughout this proof, $K$ is a positive constant that does not depend on $T$ and that is susceptible to change from one line to the other. We also set $v_0=v_p$.\\
		    For each $i\in \llbracket 1,p\cdot d \rrbracket $ and $k\in \llbracket 1,d \rrbracket$, we define the matrix process 
		    $$Z^{k,i}_{(t,\theta,x)}= \mathds 1_{\theta \leq \lambda^k_t}\mathds 1_{k=i\div p+1}\mathds 1_{t\leq v_{i \% p}T} \frac{x}{\sqrt {v_{i\%p}T}},$$ 
		    where $i\%p$ (respectively $i\div p$) is the remainder (respectively the quotient) after dividing $i$ by $p$.
		    It is possible to write the set of integers from $1$ to $p\cdot d$ as a partition of $d$ disjoint intervals $S_1 \cup \cdots \cup S_d$ each of cardinal $p$. Thus for any $n=1,\cdots,d$, $n=i\div p+1$ if and only if $i\in S_n$.\\
		    We assume $i \in S_n$ and compute the $i-$th component of the divergence of $Z$ 
		    \begin{align*}
		    \left (\delta (Z) \right )^i &=\sum_{k=1}^d \delta ^k(Z^{ik}),\\
		    &=\delta^{n}(Z^{ni}),\\
		    &=\int_{\R_+^3} Z^{ni}_{(t,\theta,x)}\left (N^n(\d t ,\d \theta ,\d x) -\d t \d \theta  \nu^n(\d x)\right),\\
		    &=\frac{1}{\sqrt{v_{i\%p}T}}\int_0^{v_{i\%p}T} \int_{\R_+^2} x\mathds 1_{\theta \leq \lambda ^n_t}  \left (N^n(\d t ,\d \theta ,\d x) -\d t \d \theta  \nu^n(\d x)\right),\\
		    &=\Gamma^i_T.
		    \end{align*}
			Using the second equality of Theorem \ref{th:Anthony}, the $d_2$ distance is bounded by 
			\begin{align*}
			d_2(\boldsymbol{ \Gamma}_T,\boldsymbol{ G})\leq & \|\hat C^{-1}\|_{op} \|\hat C\|_{op}^{1/2}\sum_{i,j=1}^{pd} \E \left [\left |\hat C_{ij}-\sum_{k=1}^d\int _{\R^2_+}\lambda^k_t  \mathds 1_{t\leq v_{i\% p}T} \mathds 1_{k=i\div p+1}\frac{x}{\sqrt{v_{i\% p} T} }D^k_{(t,\lambda^k_t,x)}  \Gamma^j_T  \mathrm d t \nu^k(\mathrm \d x) \right |\right ]\\
			&+\frac{\sqrt{2\pi}}{8}\|\hat C^{-1}\|_{op}^{3/2} \|\hat C\|_{op}  \sum_{i=1}^{pd} \sum_{k=1}^d \E \left[ \int _{\R^2_+}\lambda^k_t  \mathds 1_{t\leq v_{i\% p} T} \mathds 1_{k=i\div p+1}\frac{x}{\sqrt {v_{i\% p} T} } \left \| D^k_{(t,\lambda^k_t,x)}  \boldsymbol{ \Gamma}_T\right \|^2\mathrm d t  \nu^k(\mathrm \d x)\right],\\
			\leq & \|\hat C^{-1}\|_{op} \|\hat C\|_{op}^{1/2}\sum_{n_1,n_2=1}^{d}\sum_{i \in S_{n_1}}\sum_{j \in S_{n_2}} \E \left [\left |\hat C_{ij}-\int _{\R_+}\int_0^{v_{i\%p}T}\lambda^{n_1}_t   \frac{x}{\sqrt{v_{i\% p} T} }D^{n_1}_{(t,\lambda^{n_1}_t,x)}  \Gamma^j_T  \mathrm d t \nu^{n_1}(\mathrm \d x) \right |\right ]\\
			&+\frac{\sqrt{2\pi}}{8}\|\hat C^{-1}\|_{op}^{3/2} \|\hat C\|_{op}  \sum_{n_1=1}^{d}\sum_{i\in S_{n_1}} \E \left[ \int _{\R_+}\int_0^{v_{i\% p} T}\lambda^{n_1}_t  \frac{x}{\sqrt {v_{i\% p} T} } \left \| D^{n_1}_{(t,\lambda^{n_1}_t,x)}  \boldsymbol{ \Gamma}_T\right \|^2\mathrm d t  \nu^{n_1}(\mathrm \d x)\right],\\
			\leq & \|\hat C^{-1}\|_{op} \|\hat C\|_{op}^{1/2}\sum_{n_1,n_2=1}^{d}\sum_{i \in S_{n_1}}\sum_{j \in S_{n_2}} \E \left [\left |\hat C_{ij}-\int _{\R_+}\int_0^{v_{i\%p}T}\lambda^{n_1}_t   \frac{x}{\sqrt{v_{i\% p} T} }D^{n_1}_{(t,\lambda^{n_1}_t,x)}  \Gamma^j_T  \mathrm d t \nu^{n_1}(\mathrm \d x) \right |\right ]\\
			&+\frac{\sqrt{2\pi}}{8}\|\hat C^{-1}\|_{op}^{3/2} \|\hat C\|_{op}  \sum_{n_1,n_2=1}^{d}\sum_{i\in S_{n_1}} \sum_{j \in S_{n_2}}\E \left[ \int _{\R_+}\int_0^{v_{i\% p} T}\lambda^{n_1}_t  \frac{x}{\sqrt {v_{i\% p} T} } \left | D^{n_1}_{(t,\lambda^{n_1}_t,x)}   \Gamma_T^j\right |^2\mathrm d t  \nu^{n_1}(\mathrm \d x)\right],\\
			\leq & K\sum_{n_1,n_2=1}^{d}\sum_{i\in S_{n_1}} \sum_{j \in S_{n_2}} \left ( A^{i,j}_1+ A^{i,j}_2 \right ),
			\end{align*}
			where 
			$$A^{i,j}_1=\E \left [\left |\hat C_{ij}-\int _{\R_+}\int_0^{v_{i\%p}T}\lambda^{n_1}_t   \frac{x}{\sqrt{v_{i\% p} T} }D^{n_1}_{(t,\lambda^{n_1}_t,x)}  \Gamma^j_T  \mathrm d t \nu^{n_1}(\mathrm \d x) \right |\right ],$$
			and 
			$$A^{i,j}_2=\E \left [ \int _{\R_+}\int_0^{v_{i\% p} T}\lambda^{n_1}_t  \frac{x}{\sqrt {v_{i\% p} T} } \left | D^{n_1}_{(t,\lambda^{n_1}_t,x)}   \Gamma_T^j\right |^2\mathrm d t  \nu^{n_1}(\mathrm \d x)\right ].$$
			In both terms we have dependence on $n_1$ and $n_2$ which are respectively functions of $i$ and $j$. For $i,j=1,\cdots,d$. we will treat each term separately.\\
			\textbf{Term $A^{i,j}_1$}\\\\
			First we start by computing the Malliavin's derivative in the $i-$th direction of $\Gamma^j_T$. By linearity of the derivative operator 
			\begin{align*}
			D^{n_1}_{(t,\lambda^{n_1}_t,x)}  \Gamma^j_T  &=D^{n_1}_{(t,\lambda^{n_1}_t,x)} F^{n_2}_{v_{j\%p}T},\\
			&=\frac{1}{\sqrt {v_{j\%p}T}} D^{n_1}_{(t,\lambda^{n_1}_t,x)}\left (L^{n_2}_{v_{j\%p}T}-m^{j\%p}\int_0^{v_{j\%p}T} \lambda^{n_2}_s \d s \right ),
			\end{align*}
				which yields using Proposition $\ref{prop:DescDH}$
				\begin{align*}
				D^{n_1}_{(t,\lambda^{n_1}_t,x)}  \Gamma^j_T &=\frac{\mathds 1_{t\leq v_{j\%p}T}}{\sqrt {v_{j\%p}T}} \left (x \mathds 1_{n_1=n_2}+\hat L^{n_2,n_1,t,x}_{v_{j\%p}T} -m^{j\%p}\int_t^{v_{j\%p}T} \hat \lambda^{n_2,n_1,t,x}_s \d s\right ),\\
				&= \frac{\mathds 1_{t\leq v_{j\%p}T}}{\sqrt {v_{j\%p}T}} \left (x \mathds 1_{n_1=n_2}+\hat M^{n_2,n_1,t,x}_{v_{j\%p}T}\right ).
				\end{align*}
				Thus
				\begin{align*}
				A^{i,j}_1=&\E \left [\left |\hat C_{ij}-\int _{\R_+}\int_0^{v_{i\%p}T}\lambda^{n_1}_t   \frac{x}{\sqrt{v_{i\% p} T} } \frac{\mathds 1_{t\leq v_{j\%p}T}}{\sqrt {v_{j\%p}T}} \left (x \mathds 1_{n_1=n_2}+\hat M^{n_2,n_1,t,x}_{v_{j\%p}T}\right )  \mathrm d t \nu^{n_1}(\mathrm \d x) \right |\right ],\\
				=&\E \left [\left |\hat C_{ij}-\frac{1}{\sqrt{v_{i\% p}v_{j\% p} } T} \int _{\R_+}\int_0^{(v_{i\%p}\wedge v_{j\%p})T}x\lambda^{n_1}_t    \left (x \mathds 1_{n_1=n_2}+\hat M^{n_2,n_1,t,x}_{v_{j\%p}T}\right )  \mathrm d t \nu^{n_1}(\mathrm \d x) \right |\right ],\\
				\leq&\E \left [\left |\hat C_{ij}-\frac{\mathds 1_{n_1=n_2}}{\sqrt{v_{i\% p}v_{j\% p} } T} \int _{\R_+}\int_0^{(v_{i\%p}\wedge v_{j\%p})T}x^2\lambda^{n_1}_t \mathrm d t \nu^{n_1}(\mathrm \d x) \right |\right ]\\ &+ \frac{1}{\sqrt{v_{i\% p}v_{j\% p} } T}\E \left [\left | \int _{\R_+}\int_0^{(v_{i\%p}\wedge v_{j\%p})T}x\lambda^{n_1}_t\hat M^{n_2,n_1,t,x}_{v_{j\%p}T} \mathrm d t \nu^{n_1}(\mathrm \d x) \right | \right ],\\
				\leq&\E \left [\left |\hat C_{ij}-\frac{\mathds 1_{n_1=n_2}}{\sqrt{v_{i\% p}v_{j\% p} } T} \int _{\R_+}\int_0^{(v_{i\%p}\wedge v_{j\%p})T}x^2\E[\lambda^{n_1}_t] \mathrm d t \nu^{n_1}(\mathrm \d x) \right |\right ]\\ 
				&+\frac{\mathds 1_{n_1=n_2}}{\sqrt{v_{i\% p}v_{j\% p} } T} \int _{\R_+}x^2 \nu^{n_1}(\d x)\E \left [\left | \int_0^{(v_{i\%p}\wedge v_{j\%p})T} \lambda^{n_1}_t-\E[\lambda^{n_1}_t] \d t \right |\right ] \\
				 &+\frac{1}{\sqrt{v_{i\% p}v_{j\% p} } T}\E \left [\left | \int _{\R_+}\int_0^{(v_{i\%p}\wedge v_{j\%p})T}x\lambda^{n_1}_t\hat M^{n_2,n_1,t,x}_{v_{j\%p}T} \mathrm d t \nu^{n_1}(\mathrm \d x) \right | \right ],\\
				 \leq & A^{i,j}_{1,1}+\mathds 1_{n_1=n_2} A^{i,j}_{1,2}+A^{i,j}_{1,3}.
				\end{align*}
				We start with the term $A^{i,j}_{1,1}$. If $i$ and $j$ are not in the same interval $S_{n_1}$ (\textit{i.e.} $n_1=n_2$), $\hat C_{ij}=0$, thus $A^{i,j}_{1,1}=0$. Note that due to symmetry arguments $i$ and $j$ are exchangeable, which allows us to assume that $i$ and $j$ are in the same interval and that $i\leq j$ which is equivalent to $i\%p \leq j\%p$. In this case, we have
				$$\hat C_{ij}=C^{n_1}_{i\%pj\%p}= \int_{\R_+}x^2 \nu^{n_1} (\d x) \sqrt {\frac{v_{i\%p}}{v_{j\%p}}} \left [ \left (B-A\diag (m^1,\cdots,m^d) \right )^{-1}B \boldsymbol{ \mu}\right ]^{n_1}.$$ 

According to Lemma \ref{lemma:expectation} it is possible to put the intensity's expectation under the form 
$$\E[\boldsymbol{\lambda}_t]=\left ( B -A\diag (m^1,\cdots,m^d)\right )^{-1}B\boldsymbol{\mu}+ Qe^{-tV}\boldsymbol{\mu},$$
where $Q$ and $V$ are matrices such that all the eigen-values of $V$ are positive. Thus 
\begin{align*}
\frac{1}{\sqrt{v_{i\% p}v_{j\% p} }T}\int_0^{v_{i\%p}T} \E[\boldsymbol{\lambda}_t] \d t &= \sqrt {\frac{v_{i\%p}}{v_{j\%p}}} \left ( B -A\diag (m^1,\cdots,m^d)\right )^{-1}B\boldsymbol{\mu}+\frac{1}{\sqrt{v_{i\% p}v_{j\% p} }T}Q \int_0^{v_{i\%p}T} e^{-Vt}\d t\boldsymbol{\mu},\\
&= \sqrt {\frac{v_{i\%p}}{v_{j\%p}}}\left ( B -A\diag (m^1,\cdots,m^d)\right )^{-1}B\boldsymbol{\mu}+\frac{1}{\sqrt{v_{i\% p}v_{j\% p} }T}Q V ^{-1}\left (I_d-e^{-VTv_{i\%p}} \right )\boldsymbol{\mu},\\
\end{align*}
	which means that
\begin{align*}
A^{i,j}_{1,1} =& \E \left [  \left |\hat C_{ij}-\frac{1}{\sqrt{v_{i\% p}v_{j\% p} } T}\int_{{\R_+}}x^2 \nu^{n_1}(\d x) \int_0^{v_{i\%p}T} \E[\lambda^{n_1}_t] \d t \right | \right ],\\
\leq& \E \left[\left |\hat C_{i j}-  \int_{{\R_+}}x^2 \nu^{n_1}(\d x)\left[\big(B-A\diag(m^1,\cdots,m^d)\big)^{-1}B\boldsymbol{\mu}\right]^{n_1}\right | \right ] \\ &+\E \left [ \left |\int_{{\R_+}}x^2 \nu^{n_1}(\d x)\frac{1}{\sqrt{v_{i\% p}v_{j\% p} }T}\left [Q  V ^{-1}\left (I_d-e^{-VTv_{i\%p}} \right )\boldsymbol{\mu}\right ]^{n_1} \right | \right ],\\
=&\int_{{\R_+}}x^2 \nu^{n_1}(\d x)\frac{1}{\sqrt{v_{i\% p}v_{j\% p} } T}\E \left[\left | \left [Q V ^{-1}\left (I_d-e^{-VTv_{i\%p}} \right )\boldsymbol{\mu}\right ]^{n_1} \right | \right ],\\
=&O\left(\frac{1}{T}\right) \quad \text {because } e^{-VTv_{i\%p}} \longrightarrow 0.
\end{align*}
 For the term $A^{i,j}_{1,2}$ we use the second equality of Lemma \ref{A12} and the fact that for a given vector $\boldsymbol{ x}$, $\|\boldsymbol{x}\|_{\infty}\leq \|\boldsymbol{x}\|$ where $\|.\|$ is the euclidean norm.
\begin{align*}
\E \left [\left \|\int_0^{v_{i\%p}T} \boldsymbol{ \lambda}_t-\E[\boldsymbol{ \lambda}_t] \d t \right \|\right ] &= \E \left [\left \|\int_0^{v_{i\%p}T}e^{-\big( B-A \diag (m^1,\cdots,m^d)\big) ({v_{i\%p}T}-s)} A\boldsymbol{ M}_s \d s \right \|\right ],\\
&\leq \E \left [\int_0^{v_{i\%p}T} \left \|e^{-\big( B-A \diag (m^1,\cdots,m^d)\big) ({v_{i\%p}T}-s)} A\boldsymbol{ M}_s\right \| \d s \right ],\\
&\leq \int_0^{v_{i\%p}T} \left \|e^{-\big( B-A \diag (m^1,\cdots,m^d)\big) ({v_{i\%p}T}-s)} A\right \|_{op}\E \left [ \|\boldsymbol{ M}_s\|\right ] \d s ,\\
&\leq \int_0^{v_{i\%p}T} \left \|e^{-\big( B-A \diag (m^1,\cdots,m^d)\big) ({v_{i\%p}T}-s)} A\right \|_{op}\E \left [ \|\boldsymbol{ M}_s\|^2\right ]^{1/2} \d s ,\\
\end{align*}
where the last inequality comes from Cauchy-Schwarz.\\
Keeping in mind Ito's isometry we have
\begin{align*}
\E \left [ \|\boldsymbol{ M}_s\|^2\right ] &= \E \left [\sum_{i=1}^d \left |M^i_s\right |^2\right ],\\
&=\sum_{i=1}^d \E \left [\left |M^i_s\right |^2\right ],\\
&=\sum_{i=1}^d \E \left [\left [M^i\right ]_s\right ], \quad \text{ where } \left [M^i\right ]_s ~ \text{is the quadratic variation } \\
&=\sum_{i=1}^d \int_{{\R_+}}x^2 \nu^i(\d x) \int_0^s\E \left [\lambda^i_u\right ]\d u,\\
&\leq K (s+1),
\end{align*}
which implies that $\E \left [ \|\boldsymbol{ M}_s\|^2\right ]^{1/2} \leq K (\sqrt s +1).$\\
Using the fact that the operator norm is sub-multiplicative and Lemma \ref{lemma:Jordan} we have that
\begin{align*}
\left \|e^{-\big( B-A \diag (m^1,\cdots,m^d)\big) ({v_{i\%p}T}-s)} A\right \|_{op}\leq& \left \|e^{-\big( B-A \diag (m^1,\cdots,m^d)\big) ({v_{i\%p}T}-s)} \right \| _{op}\left \| A\right \|_{op},\\
\leq & K \left (1+({v_{i\%p}T}-s)^{d-1} \right )e^{-\rho_d({v_{i\%p}T}-s)}.\\
\end{align*}
		Combining these inequalities yields 
\begin{align*}
A^{i,j}_{1,2} &=\frac{1}{\sqrt{v_{i\% p}v_{j\% p} }T}\E \left [\left | \int_0^{v_{i\%p}T} \lambda^i_t -\E[\lambda^i_t] \d t \right| \right],\\
&\leq \frac{1}{\sqrt{v_{i\% p}v_{j\% p} }T} \E \left [\left \|\int_0^{v_{i\%p}T} \boldsymbol{ \lambda}_t-\E[\boldsymbol{ \lambda}_t] \d t \right \|\right ],\\
&\leq  \frac{K}{\sqrt{v_{i\% p}v_{j\% p} }T} \int_0^{v_{i\%p}T} e^{-\rho_d({v_{i\%p}T}-s)}\left (1+({v_{i\%p}T}-s)^{d-1} \right )(\sqrt s +1)\d s, \\
&\leq O\left ( \frac{1}{\sqrt T}\right) , ~ \text {\textit{cf.} the bound on $A_{1,2}$ in \cite{hillairet2021malliavinstein}}.
\end{align*}
For the term $A^{i,j}_{1,3}$ we start by noticing that the integral with respect to $\nu^i(\d x)$ is equivalent to taking the expectation of a random variable $X$ of law $\nu^i$. Hence
\begin{align*}
A^{i,j}_{1,3}&=\E \left [ \left | \frac{1}{\sqrt{v_{i\% p}v_{j\% p} }T} \int_0^{v_{i\% p}T} \int_{{\R_+}}x \lambda^{n_1}_t \hat M ^{n_2,n_1,t,x}_{v_{j\% p}T}  \nu^i(\d x) \d t \right | \right ],\\
&=\frac{1}{\sqrt{v_{i\% p}v_{j\% p} }T} \E \left [ \left | \E_{X} \left[\int_0^{v_{i\% p}T} X \lambda^{n_1}_t \hat M ^{n_2,n_1,t,x}_{v_{j\% p}T}   \d t\right ] \right | \right ],\\
&\leq \frac{1}{\sqrt{v_{i\% p}v_{j\% p} }T} \E \left [  \E_{X} \left[\int_0^{v_{i\% p}T} X \lambda^{n_1}_t \hat M ^{n_2,n_1,t,x}_{v_{j\% p}T}   \d t\right ] ^2 \right ]^{1/2},\\
\end{align*}
and by Jensen's inequality
\begin{align*}
A^{i,j}_{1,3} & \leq \frac{1}{\sqrt{v_{i\% p}v_{j\% p} }T} \E \left [  \E_{X} \left[ \left (\int_0^{v_{i\% p}T} X \lambda^{n_1}_t \hat M ^{n_2,n_1,t,X}_{v_{j\% p}T}   \d t\right )^2\right ]  \right ]^{1/2},\\
&\leq \frac{1}{\sqrt {v_{i\% p}v_{j\% p}T}} \left (\frac{1}{T} I_T\right )^{1/2},
\end{align*}
where 
\begin{equation}
\label{eq:IT}
I_T:= \E \left [  \E_{X} \left[ \left (\int_0^{v_{i\% p}T} X \lambda^{n_1}_t \hat M ^{n_2,n_1,t,X}_{v_{j\% p}T}  \d t\right )^2\right ]  \right ].
\end{equation}
Using lemma \ref{lemma:A13} we have that 
$$I_T \leq KT,$$
and therefore 
$$A^{i,j}_{1,3} \leq O \left ( \frac{1}{\sqrt T }\right ).$$
This shows that \begin{equation} \label{eq:A1}A_1^{i,j}= O \left ( \frac{1}{\sqrt T}\right ).\end{equation}\\
\textbf{Term $A^{i,j}_2$}\\\\
Keeping in mind that 
$$A^{i,j}_2=\E \left [ \int _{\R_+}\int_0^{v_{i\% p} T}\lambda^{n_1}_t  \frac{x}{\sqrt {v_{i\% p} T} } \left | D^{n_1}_{(t,\lambda^{n_1}_t,x)}   \Gamma_T^j\right |^2\mathrm d t  \nu^{n_1}(\mathrm \d x)\right ],$$
and that 
\begin{align*}
D^{n_1}_{(t,\lambda^{n_1}_t,x)}  \Gamma^j_T &
=\frac{\mathds 1_{t\leq v_{j\%p}T}}{\sqrt {v_{j\%p}T}} \left (x \mathds 1_{n_1=n_2}+\hat M^{n_2,n_1,t,x}_{v_{j\%p}T}\right ),
\end{align*}
we have 
\begin{align*}
A^{i,j}_2&=\frac{1}{\sqrt {v_{i\%p}}v_{j\%p}T^{3/2}}\E \left [ \int_0^{v_{i\%p}T\wedge v_{j\%p}T} \int_{{\R_+}}x\lambda^{n_1}_t \left (x\mathds 1_{n_1=n_2}+ \hat M ^{n_2,n_1,t,x}_{v_{j\%p}T} \right)^2 \nu^{n_1}(\d x) \d t  \right ],\\
&\leq \frac{2}{\sqrt {v_{i\%p}}v_{j\%p}T^{3/2}}\E \left [ \int_0^{v_{i\%p}T} \int_{{\R_+}} x^3\lambda^{n_1}_t  \mathds 1_{n_1=n_2}+ x\lambda^{n_1}_t\left | \hat M ^{n_2,n_1,t,x}_{v_{j\%p}T} \right |^2 \nu^{n_1}(\d x) \d t  \right ],\\
&=\frac{2K\cdot \mathds 1_{n_1=n_2}}{T^{3/2}} \int_{\R_+} x^3 \nu^{n_1}(\d x) \int_0^{v_{i\%p}T} \E [\lambda^{n_1}_t] \d t+ \frac{2K}{T^{3/2}}\int_0^{v_{i\%p}T} \int_{{\R_+}} x\E \left [\lambda^{n_1}_t\left | \hat M ^{n_2,n_1,t,x}_{v_{j\%p}T} \right |^2 \right] \nu^{n_1}(\d x) \d t,\\
&:= A^{i}_{2,1}+A^{i,j}_{2,2}.
\end{align*}
According to Lemma \ref{lemma:expectation} and to the fact that $\big( B-A \diag (m^1,\cdots,m^d)\big)$ has only positive eigenvalues we have that $\int_0^{v_{i\%p}T} \E [\lambda^i_t] \d t \leq KT$ and hence
$$A^i_{2,1}=O \left (\frac{1}{\sqrt T} \right ).$$
For the final term it holds that 
\begin{align*}
A^{i,j}_{2,2} &=\frac{2K}{T^{3/2}}\int_0^{v_{i\%p}T} \int_{{\R_+}} x\E \left [\lambda^{n_1}_t\left | \hat M ^{n_2,n_1,t,x}_{v_{j\%p}T} \right |^2 \right] \nu^{n_1}(\d x) \d t,\\
&=\frac{K}{T^{3/2}}\int_0^{v_{i\%p}T} \int_{{\R_+}} x\E \left [\lambda^{n_1}_t\E_t \left [\left |  \hat M ^{n_2,n_1,t,x}_{v_{j\%p}T} \right |^2\right] \right] \nu^{n_1}(\d x) \d t,\\
&=\frac{K}{T^{3/2}} \int_0^{v_{i\%p}T} \int_{{\R_+}} x\E \left [\lambda^{n_1}_t\E_t \left [\int_t^{v_{j\%p}T} \hat \lambda ^{n_2,n_1,t,x}_s ds\right] \right] \nu^{n_1}(\d x) \d t.\\
\end{align*}
Solving the SDE \ref{eq:SDEtilde} with initial condition $\E_t\left [ \tilde {\boldsymbol{ \lambda}}^{n_1,t,x}_t\right ]=x\boldsymbol{ A}_{\cdot i}$ we have that $$\E_t\left [ \tilde {{ \lambda}}^{n_2,n_1,t,x}_s\right ]=x \left (e^{-\big( B-A \diag (m^1,\cdots,m^d)\big)(s-t)}\boldsymbol{ A}_{\cdot n_1}\right )^{n_2},$$
which is integrable. It follows that
\begin{align*}
A^{i,j}_{2,2} &\leq  \frac{K}{T^{3/2}} \int_0^{v_{i\%p}T} \int_{\R_+} x^2 \E\left [\lambda^{n_1}_t\right ] \nu^{n_1}(\d x) \d t,\\
&\leq  \frac{K}{T^{3/2}} \int_{\R_+} x^2\nu^{n_1}(\d x) Kv_{i\%p}T,\\
&\leq O \left (\frac{1}{\sqrt T} \right).
\end{align*}
And finally \begin{equation} \label{eq:A2} A^{i,j}_{2}=O\left (\frac{1}{\sqrt T} \right ). \end{equation}
Combining \ref{eq:A1} and \ref{eq:A2} we conclude that 
$$d_2(\boldsymbol{ \Gamma}_T,\boldsymbol{ G})=O\left ( \frac{1}{\sqrt T}\right ).$$
\end{proof}
	\begin{corollary}
		\label{th:main}
		Fix $d \in \N^*$. Let $(\boldsymbol{L}_t)_{t \geq 0}$ be a compound multivariate Hawkes process whose intensity $\boldsymbol \lambda$ follows the dynamics \ref{dynamics}. Assume Assumptions \ref{as:stab}, \ref{as:vanish} and  \ref{as:third} are in force.\\
		Set $C=\diag(\sigma^2_1,\cdots,\sigma^2_d)$ where for any $j=1,\cdots,d$
		$$\sigma^2_j= \int x^2 \nu^j(\mathrm d x) \left[\big(B-A\diag(m^1,\cdots,m^d)\big)^{-1}B\boldsymbol{\mu}\right]^j$$
		and let $\boldsymbol{G} \sim \mathcal N(0,C)$.\\
		Then there exists a constant $K>0$ independent from $T$ such that 
		$$d_2(\boldsymbol{F}_T,\boldsymbol{G})\leq \frac{K}{\sqrt T}.$$
	\end{corollary}
	\begin{proof}
		
		This is merely an application of Theorem \ref{th:multi-marginal} for $p=1$ and $v_1=1$.
	\end{proof}
	\begin{remark}
		\label{remark:coeff}
		The specific bounds $\|g\|_{Lip}\leq 1$ and $M_2(g)\leq 1$ could have been relaxed to $\|g\|_{Lip}\leq K_1$ and $M_2(g)\leq K_2$ where $K_1$ and $K_2$ are two arbitrary positive constants.
	\end{remark}
	As a final result in this section, we consider the slightly modified process 
	$$\boldsymbol{ Y}_T = \frac{\boldsymbol{ L}_T -\diag (m^1,\cdots,m^d)\int_0^T \E [\boldsymbol{ \lambda}_s]\d s}{\sqrt T}$$
	and we study its behaviour as $T$ goes to infinity.\\
	\begin{theorem}
		\label{th:Y}
		Set $\tilde C = \left (J \sqrt C \right ) \ ^t \left (J \sqrt C \right )=JC\ ^tJ$, where $J= \left ( I_d-\diag(m^1,\cdots,m^d)B^{-1}A \right )^{-1}$ and $C$ is defined in Corollary \ref{th:main}.\\
		Let $\tilde {\boldsymbol{ G}}\sim \mathcal N \left (0, \tilde C \right )$. There exists a constant $K$ that does not depend on $T$ such that 
		$$d_2(\boldsymbol{ Y}_T,\tilde {\boldsymbol{ G}})\leq \frac{K}{\sqrt T},$$
		for any $T>0$.
	\end{theorem}
	\begin{proof}
		Thanks to Lemma \ref{lemma:Y} we have that $$J^{-1}\boldsymbol{ Y}_T=\boldsymbol{ F}_T+J^{-1}\boldsymbol{ R}_T.$$
		Let $f\in\mathcal H$ and set $f_J:\boldsymbol{x} \mapsto \frac{f(J\boldsymbol{x})} {\|J\|_{op}}$, this function is clearly in $\mathcal C ^2 (\R^d)$ and $\|f_J\|_{Lip}\leq 1$ and $M_2(f_J)\leq \|J\|_{op}$.\\
		Let $\tilde {\boldsymbol {G}}\sim \mathcal N(0,\tilde C)$. We have
		\begin{align*}
		\left |\E \left [f(\boldsymbol Y_T)\right ]-\E  [f(\tilde{\boldsymbol G}) ] \right |&=\|J\|_{op}\left |\E \left [f_J(J^{-1}\boldsymbol Y_T)\right ]-\E  \big [f_J(J^{-1}\tilde{\boldsymbol G}) \big] \right |,\\
		&=\|J\|_{op}\left |\E \left [f_J(\boldsymbol{ F}_T+J^{-1}\boldsymbol{ R}_T)\right ]-\E \left [f_J({\boldsymbol G}) \right ] \right |,\\
		\end{align*} 
		where $\boldsymbol G=J^{-1}\tilde{\boldsymbol G} \sim \mathcal N(0,C)$.\\ Using a Taylor expansion, there exists a random $\boldsymbol Y^*$ such that 
		$$f_J(\boldsymbol{ F}_T+J^{-1}\boldsymbol{ R}_T)= f_J(\boldsymbol{ F}_T)+\left <\nabla f_J(\boldsymbol Y^*), J^{-1}\boldsymbol R_T\right >,$$
		hence 
		\begin{align*}
		\left |\E \left [f(\boldsymbol Y_T)\right ]-\E  [f(\tilde{\boldsymbol G}) ] \right |&=\|J\|_{op}\left |\E \left [f_J(\boldsymbol{ F}_T) \right ]-\E \left [f_J({\boldsymbol G}) \right ]+\E \left [ \left <\nabla f_J(\boldsymbol Y^*), J^{-1}\boldsymbol R_T\right > \right ]\right |,\\&\leq \|J\|_{op}\left |\E \left [f_J(\boldsymbol{ F}_T) \right ]-\E \left [f_J({\boldsymbol G}) \right ]\right |+\|J\|_{op} \left |\E \left [ \left <\nabla f_J(\boldsymbol Y^*), J^{-1}\boldsymbol R_T\right > \right ]\right |,\\
		&\leq \|J\|_{op}\sup_{g \in \mathcal H_J}\left |\E \left [g(\boldsymbol{ F}_T) \right ]-\E \left [g({\boldsymbol G}) \right ]\right |+\|J\|_{op} \left |\E \left [ \left <\nabla f_J(\boldsymbol Y^*), J^{-1}\boldsymbol R_T\right > \right ]\right |,\\
		\end{align*}
		where $\mathcal H_J=\{f \in \mathcal C ^2(\R^d) \text { such that } \|f\|_{Lip}\leq 1 \text{ and } M_2(f)\leq \|J\|_{op}\}$. Using Cauchy-Schwarz's inequality (twice) the second term is bounded as follows 
		\begin{align*}
		\left |\E \left [ \left <\nabla f_J(\boldsymbol Y^*), J^{-1}\boldsymbol R_T\right > \right ] \right |&\leq \left |\E \left [ \left \|\nabla f_J(\boldsymbol Y^*) \right \| \left \| J^{-1}\boldsymbol R_T\right \| \right ] \right |,\\
		&\leq \left |\E \left [ \left \|\nabla f_J(\boldsymbol Y^*) \right \|^2 \right ]^{1/2} \E \left [ \left \| J^{-1}\boldsymbol R_T\right \|^2\right ]^{1/2} \right |,\\
		&\leq \left |\E \left [ \left \|f_J \right \|_{Lip}^2 \right ]^{1/2} \E \left [ \left \| J^{-1} \right \|_{op}^2 \left \|\boldsymbol R_T\right \|^2\right ]^{1/2} \right |,\\
		&=\left \| J^{-1} \right \|_{op} \E \left [ \left \|\boldsymbol R_T\right \|^2\right ]^{1/2}.
		\end{align*}
		Thanks to Corollary \ref{th:main} and Remark \ref{remark:coeff} there exists a positive constant $K$ (independent from $f$ and $T$ and that can change from one line to another) such that 
		$$\left |\E \left [f(\boldsymbol Y_T)\right ]-\E  [f(\tilde{\boldsymbol G}) ] \right |\leq K \left (\frac{1}{\sqrt T} +\E \left [ \left \|\boldsymbol R_T\right \|^2\right ]^{1/2} \right ).$$
		And using Lemma \ref{lemma:second} we have that
		\begin{align*}
		\E \left [ \left \|\boldsymbol R_T\right \|^2\right ]&=\E \left [ \left \|\diag (m^1,\cdots,m^d) \left ( B-A \diag (m^1,\cdots,m^d)\right )^{-1} \frac{\E[\boldsymbol \lambda _T]-\boldsymbol \lambda _T}{\sqrt T}\right \|^2\right ],\\
		&\leq \left \|\diag (m^1,\cdots,m^d) \left ( B-A \diag (m^1,\cdots,m^d)\right )^{-1} \right \|_{op}^2\frac{1}{T}\E \left [ \left \|\E[\boldsymbol \lambda _T]-\boldsymbol \lambda _T \right \|^2\right ],\\
		&\leq \frac{K}{T} \sum_{i=1}^d \E \left [\left (\lambda^i_T-\E[\lambda^i_T]\right )^2\right ],\quad \text {and by virtue of Lemma \ref{lemma:second},}\\
		&\leq \frac{K}{T}.
		\end{align*}
		And finally 
		$$\left |\E \left [f(\boldsymbol Y_T)\right ]-\E  [f(\tilde{\boldsymbol G}) ] \right | = O\left (\frac{1}{\sqrt T} \right ),$$
		hence the result.
	\end{proof}
\subsection{Proof of Theorem \ref{th:intro}}
\label{section:proof}
We start by recalling that 
$$\boldsymbol{ Y}_T = \frac{\boldsymbol{ L}_T -\diag (m^1,\cdots,m^d)\int_0^T \E [\boldsymbol{ \lambda}_s]\d s}{\sqrt T},$$
and that
\begin{align*}
\boldsymbol{ Y}'_T&= \frac{\boldsymbol{ L}_T -\diag (m^1,\cdots,m^d)\left ( B -A\diag (m^1,\cdots,m^d)\right )^{-1}B\boldsymbol{\mu} T}{\sqrt T},\\
&= \frac{\boldsymbol{ L}_T -\diag (m^1,\cdots,m^d) V^{-1}B\boldsymbol{\mu} T}{\sqrt T},
\end{align*}
where $V= B -A\diag (m^1,\cdots,m^d).$\\
Using Lemma \ref{lemma:expectation}, we have that 
\begin{align*}
\E[\boldsymbol{\lambda}_t]=V^{-1} B\boldsymbol{\mu}+ e^{-V t}\left (I_d- V^{-1}B \right )\boldsymbol{\mu},
\end{align*}
hence
\begin{align*}
\int_0^T \E[\boldsymbol{\lambda}_t] \d t &= V^{-1} B\boldsymbol{\mu} T + \int_0^T e^{-V t} \d t\left (I_d- V^{-1}B \right )\boldsymbol{\mu},\\
&=V^{-1} B\boldsymbol{\mu} T + \left (I_d -e^{-VT} \right )V^{-1}\left (I_d- V^{-1}B \right )\boldsymbol{\mu}.
\end{align*}
We deduce that $$\boldsymbol{ Y'}_T=\boldsymbol{ Y}_T - \boldsymbol{ R'}_T,$$
with $\boldsymbol{ R'}_T= \frac{\diag(m^1,\cdots,m^d) \left (I_d -e^{-VT} \right )\left (V^{-1}- V^{-2}B \right )\boldsymbol{\mu}}{\sqrt T}$.\\
Let $\tilde {\boldsymbol{ G}}\sim \mathcal N \left (0, \tilde C \right )$ as defined in Theorem \ref{th:intro} and let $f\in \mathcal H$. Using a Taylor expansion, we have for some $\boldsymbol{ X}$
$$f(\boldsymbol{ Y'}_T)-f(\tilde {\boldsymbol{ G}})=f(\boldsymbol{ Y}_T)-f(\tilde {\boldsymbol{ G}})-\left <\nabla f (\boldsymbol{ X}) , \boldsymbol{ R'}_T \right >.$$
Since $\sup_{\boldsymbol{x}\in \mathbb R^d}\| \nabla f (\boldsymbol{x})\|_2 \leq 1$ and $\boldsymbol{ Y'}_T=O\left (\frac{1}{\sqrt T} \right )$, we deduce that 
$$d_2(\boldsymbol{ Y'}_T,\tilde{\boldsymbol{ G}} )\leq d_2(\boldsymbol{ Y}_T,\tilde{\boldsymbol{ G}} )+O\left ( \frac{1}{\sqrt T}\right ),$$
which yields using Theorem \ref{th:Y}
$$d_2(\boldsymbol{ Y'}_T,\tilde{\boldsymbol{ G}} )= O\left ( \frac{1}{\sqrt T}\right ).$$
	\section{Lemmata}
	\begin{lemma}
		\label{lemma:Jordan} Set $V=B-A \diag (m^1,\cdots,m^d)$.
		Assume that Assumption \ref{as:vanish} is in force. Then there are positive constants $K$ and $\rho_d$ such that 
		$$\left \| e^{-tV}\right \|_{op}\leq K(1+t^{d-1})e^{-\rho_d t}$$
		for any $t\geq0$.
	\end{lemma}
	\begin{proof}
		First, we call $V$'s eigenvalues  $\rho_1\geq\cdots \geq \rho_d$ and we recall that they are positive. Using Jordan-Chevalley's decomposition we can write 
		$$V=P\diag (\rho_1,\cdots,\rho_d)P^{-1}+V_{nil}$$
		where $V_{nil}$ is a nilpotent matrix that commutes with $P\diag (\rho_1,\cdots,\rho_d)P^{-1}$. Let $t\geq0$, taking the exponential yields
		\begin{align*}
		e^{-tV}&=e^{-tP\diag (\rho_1,\cdots,\rho_d)P^{-1}-tV_{nil}},\\
		&=e^{-tP\diag (\rho_1,\cdots,\rho_d)P^{-1}} e^{-tV_{nil}},\\
		&=P\diag(e^{-\rho_1 t},\cdots,e^{-\rho_d t})P^{-1}\sum_{j=1}^{d-1} t^j \frac{(-V_{nil})^j}{j!}.
		\end{align*}
		Since the operator norm is sub-multiplicative and using the triangular inequality  
		\begin{align*}
		\left \|e^{-tV} \right \|_{op} &\leq \|P\|_{op} \|P^{-1}\|_{op} \left \|\diag(e^{-\rho_1 t},\cdots,e^{-\rho_d t}) \right \|_{op}\sum_{j=1}^{d-1} t^j \frac{\|V_{nil}\|_{op}^j}{j!},\\
		&\leq K \left \|\diag(e^{-\rho_1 t},\cdots,e^{-\rho_d t}) \right \|_{op}\left ( 1+t^{d-1}\right ),\\
		&\leq K \left \|\diag(e^{-\rho_1 t},\cdots,e^{-\rho_d t}) \right \|_{\infty}\left ( 1+t^{d-1}\right ),\\
		\end{align*}
		where the last inequality comes from the fact that all norms are equivalent in finite dimension.
	\end{proof}
	\begin{lemma}
		\label{lemma:expectation}
		If $\boldsymbol{\lambda}$ follows the dynamics \ref{dynamics} then for each $t$ 
		$$\E[\boldsymbol{\lambda}_t]=\left ( B -A\diag (m^1,\cdots,m^d)\right )^{-1} B\boldsymbol{\mu}+ e^{-\big( B-A \diag (m^1,\cdots,m^d)\big) t}\left (I_d- \big( B-A \diag (m^1,\cdots,m^d)\big)^{-1}B \right )\boldsymbol{\mu}.$$
	\end{lemma}
	\begin{proof}
		First we prove that if the kernels are exponential, then $\boldsymbol{\lambda}$ is a Markov process that follows an SDE. By multiplying \ref{dynamics} by $e^{B t}$ we get 
		\begin{align*}
		e^{B t}\boldsymbol{\lambda}_t&=e^{B t}\boldsymbol{\mu}+\int_{[0,t)}e^{B s} A \d \boldsymbol{L}_s,\quad \text{and by differentiating, }\\
		e^{B t}\left (\d \boldsymbol{\lambda}_t+B \boldsymbol{\lambda}_t \d t\right )&= e^{B t} B \boldsymbol{\mu} \d t + e^{B t} A \d \boldsymbol{L}_t.
		\end{align*}
		Hence the SDE
		\begin{equation}
		\label{SDE}
		\d\boldsymbol{ \lambda}_t = B\left (\boldsymbol{\mu}- \boldsymbol{ \lambda}_t \right)\d t + A \d\boldsymbol{L}_t.
		\end{equation}
		By taking the expected value of \ref{SDE} and keepind in mind that $\d\E [\boldsymbol{L}_s]=\diag (m^1,\cdots,m^d)\E[\boldsymbol{\lambda}_s] \d s$, the intensity's expectation is the solution of the multivariate ODE 
		\begin{equation}
		\label{ODE}
		\d\E[\boldsymbol{\lambda}_t]=B \boldsymbol{\mu}\d t-\left ( B -A\diag (m^1,\cdots,m^d)\right ) \E [\boldsymbol{\lambda}_t]\d t,
		\end{equation}
		with initial condition $\E[\boldsymbol{\lambda}_0]=\mu$. Assumption \ref{as:vanish} guarantees that all the eigenvalues of $\left ( B -A\diag (m^1,\cdots,m^d)\right )$ are positive, thus the matrix is invertible and we get the result.
	\end{proof}
	\begin{lemma}
		\label{A12}
		The difference between the intensity and its expected value is
		$$\boldsymbol{ \lambda}_t-\E[\boldsymbol{ \lambda}_t]=\int_{[0,t)}e^{-\big( B-A \diag (m^1,\cdots,m^d)\big) (t-s)} A \d \boldsymbol{M}_s,$$
		where $$\boldsymbol{ M}_t=\boldsymbol{L}_t-\int_0^t \diag (m^1,\cdots,m^d) \boldsymbol{\lambda}_s \d s.$$
		Its integral with respect to time is
		$$\int_0^T \boldsymbol{ \lambda}_t-\E[\boldsymbol{ \lambda}_t] \d t = \int_0^Te^{-\big( B-A \diag (m^1,\cdots,m^d)\big) (T-s)} A\boldsymbol{ M}_s ds. $$
	\end{lemma}
	\begin{proof}
		By taking the difference between \ref{SDE} and \ref{ODE} we can verify that $\boldsymbol{ \lambda}-\E[\boldsymbol{ \lambda}]$ is a solution of the SDE
		$$\d \left ( \boldsymbol{ \lambda}_t-\E[\boldsymbol{ \lambda}_t]\right )=-\big( B-A \diag (m^1,\cdots,m^d)\big)\left ( \boldsymbol{ \lambda}_t-\E[\boldsymbol{ \lambda}_t]\right )+A\d \boldsymbol{ M}_t$$
		with the initial condition $\boldsymbol{ \lambda}_0-\E[\boldsymbol{ \lambda}_0]=\boldsymbol 0.$ Solving the SDE (using variation of parameters) yields
		$$\boldsymbol{ \lambda}_t-\E[\boldsymbol{ \lambda}_t]=\int_{[0,t)}e^{-\big( B-A \diag (m^1,\cdots,m^d)\big) (t-s)} A \d \boldsymbol{M}_s.$$
		For the second equality, we start by taking the integral with respect to time until the instant $T$
		\begin{align*}
		\int_0^T \boldsymbol{ \lambda}_t-\E[\boldsymbol{ \lambda}_t] \d t &=\int_0^T\int_{[0,t)}e^{-\big(B-A \diag (m^1,\cdots,m^d)\big) (t-s)} A \d \boldsymbol{M}_s\d t,\\
		&=\int_0^T\int_0^T \mathds 1_{s<t} e^{-\big( B-A \diag (m^1,\cdots,m^d)\big) (t-s)} A \d \boldsymbol{M}_s\d t, \quad \text{and using Fubini's identity}\\
		&=\int_0^T\int_s^T 
		e^{-\big( B-A \diag (m^1,\cdots,m^d)\big) (t-s)}\d t A \d \boldsymbol{M}_s,\\
		&= \int_0^T\int_0^{T-s} 
		e^{-\big( B-A \diag (m^1,\cdots,m^d)\big) u}\d u A \d \boldsymbol{M}_s,\quad \text{using a change of variables}\\
		&= \int_0^T\Phi(T-s) A \d \boldsymbol{M}_s,
		\end{align*}
		where $\Phi$ is the anti-derivative of $u\to e^{-\big( B-A \diag (m^1,\cdots,m^d)\big) u}$  that vanishes at zero. In absence of common jumps, the integration by parts formula is
		$$0=\Phi(0)A\boldsymbol{ M}_T-\Phi(T)A\boldsymbol{ M}_0=\int_0^T \d \left (\Phi(T-s) \right )A \boldsymbol{ M}_s +\Phi(T-s) A \d \boldsymbol{M}_s,$$
		hence
		$$\int_0^T \boldsymbol{ \lambda}_t-\E[\boldsymbol{ \lambda}_t] \d t = \int_0^Te^{-\big( B-A \diag (m^1,\cdots,m^d)\big) (T-s)} A\boldsymbol{ M}_s \d s. $$
	\end{proof}
	\begin{lemma}
		\label{lemma:second}
		For any fixed $n,m$ in $\llbracket 1,d \rrbracket$ we have for any $T \geq 0$
		$$\E \left [\left ( \lambda^n_T-\E [\lambda^n_T]\right  ) \left ( \lambda^m_T -\E [\lambda^m_T]\right ) \right ]\leq K $$
		where $K$ is a positive constant independent from $T$. 
	\end{lemma}
	\begin{proof}
		Using the first equality of Lemma \ref{A12} we have 
		$$\left ( \boldsymbol{ \lambda}_T-\E[\boldsymbol{ \lambda}_T] \right )\left ( \boldsymbol{ \lambda}_T-\E[\boldsymbol{ \lambda}_T] \right )^\top =\int_{[0,T)}e^{-\big( B-A \diag (m^1,\cdots,m^d)\big) (T-s)} A \d \boldsymbol{M}_s\int_{[0,T)} \d\boldsymbol{ M}_s^\top  A^\top e^{-\big( B-A \diag (m^1,\cdots,m^d)\big)^\top (T-s)},  $$
		
		which yields using the multivariate Ito isometry
		$$\E \left [\left ( \boldsymbol{ \lambda}_T-\E[\boldsymbol{ \lambda}_T] \right )\left ( \boldsymbol{ \lambda}_T-\E[\boldsymbol{ \lambda}_T] \right )^\top \right ]=\int_{[0,T)}e^{-\big( B-A \diag (m^1,\cdots,m^d)\big) (T-s)} A \d\E \left [\left [ \boldsymbol{ M}\right ]_s\right ]  A^\top e^{-\big( B-A \diag (m^1,\cdots,m^d)\big)^\top (T-s)},$$
		where $\ d\left [ \boldsymbol{ M}\right ]$ is the quadratic variation matrix infinitesimal growth
		\begin{align*}
		\d\big (\left [ \boldsymbol{ M}\right ]_s\big)_{ij}&= \d\left [ M^i,M^j\right ]_s,\\
		&=\mathds 1_{i\neq j} \d\left [ M^i\right ]_s,\\
		&=\mathds 1_{i\neq j} \E \left [ \lambda ^i_s\right ] \d s.\\
		\end{align*}
		Since the expected value of $\lambda^i_s$ is bounded by a constant independent from $T$ and since all of $\big( B-A \diag (m^1,\cdots,m^d)\big)$'s eigen-values are positive we obtain the result.
	\end{proof}
	\begin{lemma}
		\label{lemma:A13}
		Assume that Assumptions \ref{as:third} and \ref{as:stab} are is force. For $i,j \in \llbracket 1,d \rrbracket$ let $I_T$ be the quantity defined in \ref{eq:IT}, then we have
		$$\frac{1}{T}I_T \leq K.$$
	\end{lemma}
	\begin{proof}
		Let $m \leq n$ be two integers in $\llbracket 1,p \rrbracket$ and $i,j$ in $\llbracket 1,d \rrbracket$. 
		We start by expanding \ref{eq:IT}
		\begin{align*}
		I_T&=\E \left [  \E_{X} \left[ \left (\int_0^{v_m T} X \lambda^i_t \hat M ^{j,i,t,X}_{v_n T}   \d t\right )^2\right ]  \right ],\\
		&=2 \E \left [  \E_{X} \left[  \int_0^{v_m T} \int_0^t X^2 \lambda^i_t \hat M ^{j,i,t,X}_{v_n T}  \lambda^i_s \hat M ^{j,i,s,X}_{v_n T} \d s \d t \right ]  \right ],\\
		&=2 \int_{{\R_+}} \int_0^{v_m T} \int_0^t x^2 \E \left [\lambda^i_t \lambda^i_s \E_t \left [ \hat M ^{j,i,t,x}_{v_n T} \hat M ^{j,i,s,x}_{v_n T}\right ]\right ] \d s \d t \nu^i{\d x},\\
		\end{align*}
		where we recall that $\hat M ^{j,i,t,x}_{v_n T}=\hat L ^{j,i,t,x}_{v_n T} -m^j\int_t^{v_n T} \hat \lambda^{j,i,t,x}_s ds$ is a martingale. Thus the product's expectation is
		\begin{align*}
		\E_t \left [ \hat M ^{j,i,t,x}_{v_n T} \hat M ^{j,i,s,x}_{v_n T}\right ]&= \E_t \left [ \hat M ^{j,i,t,x}_{v_n T} \left (\hat M ^{j,i,s,x}_{v_n T} - \hat M ^{j,i,s,x}_t\right)\right ] + \E _t\left [ \hat M ^{j,i,t,x}_{v_n T} \hat M ^{j,i,s,x}_t \right ],\\
		\end{align*}
		the last term vanishes since for any $t\leq v_mT \leq v_nT$, $\E _t\left [ \hat M ^{j,i,t,x}_{v_n T} \hat M ^{j,i,s,x}_t \right ]= \hat M ^{j,i,s,x}_t  \E _t\left [ \hat M ^{j,i,t,x}_{v_n T}  \right ]=0$. Thus 
		
		\begin{align*}
		\E_t \left [ \hat M ^{j,i,t,x}_{v_n T} \hat M ^{j,i,s,x}_{v_n T}\right ]&=\E _t\left [ \int_t^{v_n T} \int_{{\R_+}} \int_{\R_+}y^2 \mathds 1_{\theta\leq \hat \lambda ^{j,i,t,x}_u}\mathds 1_{\theta\leq \hat \lambda ^{j,i,s,x}_u} \d \theta \nu^j(\d y) \d u\right],\\
		&=\E_t \left [ \int_t^{v_n T} \int_{{\R_+}} \int_{\R_+}y^2 \mathds 1_{\theta\leq \min(\hat \lambda ^{j,i,t,x}_u,\hat\lambda ^{j,i,s,x}_u)}\d \theta \nu^j(\d y) \d u\right],\\
		&=\int_{{\R_+}} y^2 d\nu^j(y) \int_t^{v_n T}\E_t \left [  \min(\hat \lambda ^{j,i,t,x}_u,\hat\lambda ^{j,i,s,x}_u)\right]\d u,\\
		&\leq K \int_t^{v_n T}\E_t \left [   \hat\lambda ^{j,i,s,x}_u\right]\d u.
		\end{align*}
		As shown in remark \ref{remark:tilde}, $\hat{\boldsymbol \lambda} ^{i,s,x}_u$ has the same dynamics as a Hawkes process $\tilde{\boldsymbol \lambda} ^{i,s,x}_u$ that satisfies -if we follow the lines of the proof of Lemma \ref{lemma:expectation}- the SDE
		\begin{equation}
		\label{eq:SDEtilde}
		\d\tilde {\boldsymbol{ \lambda}}_u ^{i,s,x}= -B \tilde {\boldsymbol{ \lambda}}_u^{i,s,x} \d u + A \d\tilde {\boldsymbol{L}}_u^{i,s,x}
		\end{equation}
		whose solution yields
		\begin{align*}
		\E_t \left [   \tilde{\boldsymbol \lambda} ^{i,s,x}_u\right ]&=e^{-V (u-t)}\tilde{\boldsymbol \lambda} ^{i,s,x}_t,\\
		\end{align*}
		where $V=B-A \diag (m^1,\cdots,m^d)$ whose eigen-values $\rho_1,\cdots,\rho_d$ are positive.
		Hence
		\begin{align*}
		\E_t \left [ \hat M ^{j,i,t,x}_{v_n T} \hat M ^{j,i,s,x}_{v_n T}\right ] &\leq K \left [ \int_t^{v_n T}e^{-V (u-t)} \tilde{\boldsymbol \lambda} ^{i,s,x}_t \d u\right ]^j,\\
		&=K \left [ V^{-1} \left (I_d-e^{-V({v_n T}-t)} \right )\tilde{\boldsymbol \lambda} ^{i,s,x}_t \right ]^j.\\
		\end{align*}
		By plugging this inequality in $I_T$'s expression we get 
		\begin{align*}
		I_1 &\leq K \int_{{\R_+}} \int_0^{v_mT} \int_0^t x^2 \E \left [\lambda^i_t \lambda^i_s \left [ V^{-1} \left (I_d-e^{-V({v_n T}-t)} \right )\tilde{\boldsymbol \lambda} ^{i,s,x}_t \right ]^j\right ] \d s \d t \nu^i(\d x),\\
		&=K \int_{{\R_+}} \int_0^{v_mT} \int_0^t x^2  \left [ V^{-1} \left (I_d-e^{-V({v_n T}-t)} \right ) \E \left [\lambda^i_t \lambda^i_s\tilde{\boldsymbol \lambda} ^{i,s,x}_t \right ]\right ]^j \d s \d t\nu^i(\d x),\\
		&=K \int_{{\R_+}} \int_0^{v_mT} \int_0^t x^2  \left [ V^{-1} \left (I_d-e^{-V({v_n T}-t)} \right ) \E \left [\lambda^i_s \E_s\left [\lambda^i_t\tilde{\boldsymbol \lambda} ^{i,s,x}_t \right ] \right ]\right ]^j \d s \d t \nu^i(\d x),\\
		&=K \int_{{\R_+}} \int_0^{v_mT} \int_0^t x^2  \left [ V^{-1} \left (I_d-e^{-V({v_n T}-t)} \right ) \E \left [\lambda^i_s \E_s\left [\lambda^i_t\tilde{\boldsymbol \lambda} ^{i,s,x}_t \right ] \right ]\right ]^j \d s \d t \nu^i(\d x),\\
		\end{align*}
		since $\tilde{\boldsymbol \lambda} ^{i,s,x}_t$ starts at $s$, $\E_s\left [\lambda^i_t\tilde{\boldsymbol \lambda} ^{i,s,x}_t \right ]= \E_s\left [ \lambda^i_t\right ]\E \left [ \tilde{\boldsymbol \lambda} ^{i,s,x}_t\right ]$. By solving the expectation value of the SDE \ref{eq:SDEtilde} with the initial condition $\E \left [ \tilde{\boldsymbol \lambda} ^{i,s,x}_s\right ]=\boldsymbol{A}_{.i}x$ we get 
		\begin{align*}
		\E \left [ \tilde{\boldsymbol \lambda} ^{i,s,x}_t\right ]&=e^{-V (t-s)}  \boldsymbol{A}_{.i}x,\\
		\end{align*}
		which yields after being plugged in the last inequality 
		\begin{align*}
		I_T \leq& K \int_{{\R_+}} \int_0^{v_mT} \int_0^t x^2  \left [ V^{-1} \left (I_d-e^{-V({v_n T}-t)} \right )e^{-V (t-s)}\boldsymbol{A}_{.i}x \right ]^j \E \left [\lambda^i_s \E_s\left [ \lambda^i_t\right ] \right ] \d s \d t \nu^i(\d x),\\
		\leq& K \int_{{\R_+}} \int_0^{v_mT} \int_0^t x^3 \left [ V^{-1} \left (I_d-e^{-V({v_n T}-t)} \right )e^{-V (t-s)}\boldsymbol{A}_{.i} \right ]^j \E \left [\lambda^i_s  \lambda^i_t \right ] \d s \d t \nu^i(\d x).\\
		\end{align*}
		By combining Cauchy-Schwarz's inequality and Lemma \ref{lemma:second} we have that 
		$$\E \left [\lambda^i_s  \lambda^i_t \right ] \leq \E \left [|\lambda^i_s|^2\right ]^{1/2}\E \left [|\lambda^i_t|^2 \right ]^{1/2}\leq K,$$
		thus 
		\begin{align*}
		I_T \leq& K \int_{\R_+} x^3 \nu^i(\d x)\int_0^{v_mT}  \left [ V^{-1} \left (I_d-e^{-V({v_n T}-t)} \right )
		\int_0^t e^{-V (t-s)}\d s\boldsymbol{A}_{.i} \right ]^j   \d t,\\
		\leq& K \int_{\R_+} x^3 \nu^i(\d x)\int_0^{v_mT}  \left [V^{-1} \left (I_d-e^{-V({v_n T}-t)} \right )V^{-1}\left (I_d-e^{-Vt} \right )\boldsymbol{A}_{.i} \right ]^j   \d t,\\
		\leq& K \int_0^{v_mT}  \left [V^{-2} \left (I_d-e^{-V({v_n T}-t)} \right )\left (I_d-e^{-Vt} \right )\boldsymbol{A}_{.i} \right ]^j   \d t,\quad \text{because $V$ commutes with its exponential,}\\
		=& K \int_0^{v_mT}  \left [V^{-2} \left (I_d-e^{-V({v_n T}-t)} -e^{-Vt}+ e^{-V{v_n T}}\right )\boldsymbol{A}_{.i} \right ]^j   \d t,\\
		\end{align*}
		and since $ v_mT \leq v_nT$,
		\begin{align*}
		I_T \leq& K \int_0^{v_nT}  \left \|I_d\right \|_{op}+ \left \| e^{-V({v_n T}-t)} \right \|_{op}+ \left \| e^{-Vt} \right \|_{op}+ \left \| e^{-VT} \right \|_{op}\d t,\\
		\leq& K \int_0^{v_nT} 1+ (1+({v_n T}-t)^{d-1})e^{-(\beta-\rho_d)({v_n T}-t)}+(1+t^{d-1})e^{-(\beta-\rho_d)t} +(1+{v_n T}^{d-1})e^{-(\beta-\rho_d){v_n T}}  \d t,\\
		\leq & KT.
		\end{align*} 
	\end{proof}
	\begin{lemma}
		\label{lemma:Y}
		Set $\boldsymbol{ F}_T=\frac{\boldsymbol{ L}_T -\diag(m^1,\cdots,m^d)\int_0^T \boldsymbol \lambda_t \d t}{\sqrt T}$ and $\boldsymbol{ Y}_T=\frac{\boldsymbol{ L}_T -\diag(m^1,\cdots,m^d)\int_0^T\E \left [ \boldsymbol \lambda_t \right ] \d t}{\sqrt T}$. Then we have the equality 
		\begin{equation}
		\label{eq:decomposition}
		\boldsymbol{ Y}_T=J\boldsymbol{ F}_T+\boldsymbol{ R}_T,
		\end{equation}
		where 
		$$J= \left ( I_d-\diag(m^1,\cdots,m^d)B^{-1}A \right )^{-1}$$
		and
		$$\boldsymbol{ R}_T=\diag (m^1,\cdots,m^d) \left ( B-A \diag (m^1,\cdots,m^d)\right )^{-1} \frac{\E[\boldsymbol \lambda _T]-\boldsymbol \lambda _T}{\sqrt T}.$$
	\end{lemma}
	\begin{proof}
		By taking the expected value of  SDE \ref{SDE} we have the system 
		\begin{equation*}
		\left\lbrace
		\begin{array}{l}
		\d \boldsymbol{ \lambda}_t=B (\boldsymbol \mu-\boldsymbol{ \lambda}_t)\d t+A\d \boldsymbol{ L}_t,\\\\
		\d\E\left [\boldsymbol{ \lambda}_t\right ]=B (\boldsymbol \mu-\E[ \boldsymbol{ \lambda}_t])\d t+A \diag(m^1,\cdots,m^d)\E[\boldsymbol{ \lambda}_t] \d t,
		\end{array}
		\right.
		\end{equation*}
		which yields after integrating with respect to time
		\begin{equation*}
		\left\lbrace
		\begin{array}{l}
		\boldsymbol{ \lambda}_T-\boldsymbol{\mu}=B \boldsymbol \mu T-B \int_0^T \boldsymbol{ \lambda}_t \d t+A\boldsymbol{ L}_T,\\\\
		\E\left [\boldsymbol{ \lambda}_T\right ]-\boldsymbol{\mu}=B \boldsymbol \mu T-\left (B -A \diag(m^1,\cdots,m^d)\right )\int_0^T\E[ \boldsymbol{ \lambda}_t]\d t.
		\end{array}
		\right.
		\end{equation*}
		In order to involve the quantities of interest $\boldsymbol{ F}_T$ and $\boldsymbol{ Y}_T$, we state the fact that $\diag(m^1,\cdots,m^d)$ is invertible (since $\nu^1,\cdots,\nu^d$ are suppoted by $\R_+^*$), hence 
		\begin{equation*}
		\left\lbrace
		\begin{array}{l}
		\boldsymbol{ \lambda}_T-\boldsymbol{\mu}=B \boldsymbol \mu T-B \diag(m^1,\cdots,m^d)^{-1}\diag(m^1,\cdots,m^d) \int_0^T \boldsymbol{ \lambda}_t \d t+A\d \boldsymbol{ L}_T,\\\\
		\E\left [\boldsymbol{ \lambda}_T\right ]-\boldsymbol{\mu}=B \boldsymbol \mu T-\left (B\diag(m^1,\cdots,m^d)^{-1} -A \right )\diag(m^1,\cdots,m^d)\int_0^T\E[ \boldsymbol{ \lambda}_t]\d t,
		\end{array}
		\right.
		\end{equation*}
		which yields by adding and subtracting $\boldsymbol L _T$
		\begin{equation*}
		\left\lbrace
		\begin{array}{l}
		\boldsymbol{ \lambda}_T-\boldsymbol{\mu}=B \boldsymbol \mu T+B \diag(m^1,\cdots,m^d)^{-1}\left ( \sqrt T \boldsymbol{ F}_T -\boldsymbol{ L}_T\right )+A\d \boldsymbol{ L}_T,\\\\
		\E\left [\boldsymbol{ \lambda}_T\right ]-\boldsymbol{\mu}=B \boldsymbol \mu T+\left (B\diag(m^1,\cdots,m^d)^{-1} -A \right )\left (\sqrt T \boldsymbol{ Y}_T - \boldsymbol{ L}_T \right ).
		\end{array}
		\right.
		\end{equation*}
		Subtracting the first equation from the second yields
		\begin{align*}
		\E\left [\boldsymbol{ \lambda}_T\right ]-\boldsymbol{ \lambda}_T &= \left ( B \diag(m^1,\cdots,m^d)^{-1} -A \right)\sqrt T \boldsymbol{ Y}_T-B \diag(m^1,\cdots,m^d)^{-1} \sqrt T \boldsymbol{ F}_T,\\
		&= \left ( B -A \diag(m^1,\cdots,m^d)\right )\diag(m^1,\cdots,m^d)^{-1}\sqrt T \boldsymbol{ Y}_T-B \diag(m^1,\cdots,m^d)^{-1} \sqrt T \boldsymbol{ F}_T.
		\end{align*}
		Since $\left ( B - \diag(m^1,\cdots,m^d)A\right )$ is invertible, we have
		\begin{align*}
		\left ( B -A \diag(m^1,\cdots,m^d)\right )^{-1} \frac{\E\left [\boldsymbol{ \lambda}_T\right ]-\boldsymbol{ \lambda}_T}{\sqrt T}=&\diag(m^1,\cdots,m^d)^{-1}\boldsymbol{ Y}_T\\&- \left ( B -A \diag(m^1,\cdots,m^d)\right )^{-1} B \diag(m^1,\cdots,m^d)^{-1} \boldsymbol{ F}_T,\\
		=&\diag(m^1,\cdots,m^d)^{-1}\boldsymbol{ Y}_T\\&- \left (  \diag(m^1,\cdots,m^d) -\diag(m^1,\cdots,m^d)B^{-1}A \diag(m^1,\cdots,m^d)\right )^{-1}  \boldsymbol{ F}_T.\\
		\end{align*}
		Multiplying to the left by $\diag(m^1,\cdots,m^d)$ we get
		$$\diag(m^1,\cdots,m^d)\left ( B -A \diag(m^1,\cdots,m^d)\right )^{-1} \frac{\E\left [\boldsymbol{ \lambda}_T\right ]-\boldsymbol{ \lambda}_T}{\sqrt T}=\boldsymbol{ Y}_T- \left (  I_d - \diag(m^1,\cdots,m^d)B^{-1}A\right )^{-1} \boldsymbol{ F}_T.$$
	\end{proof}
	
	\section*{Acknowledgement}
	
	I would like to thank Prof. Anthony R{\'e}veillac for his guidance and his useful critiques of this paper.


\begin{thebibliography}{10}
	\bibitem{achab2017uncovering}
	Massil Achab, Emmanuel Bacry, St{\'e}phane Ga{\i}ffas, Iacopo Mastromatteo, and
	Jean-Fran{\c{c}}ois Muzy.
	\newblock Uncovering causality from multivariate hawkes integrated cumulants.
	\newblock In {\em International Conference on Machine Learning}, pages 1--10.
	PMLR, 2017.
	
	\bibitem{bacry2013modelling}
	Emmanuel Bacry, Sylvain Delattre, Marc Hoffmann, and Jean-Fran{\c{c}}ois Muzy.
	\newblock Modelling microstructure noise with mutually exciting point
	processes.
	\newblock {\em Quantitative finance}, 13(1):65--77, 2013.
	
	\bibitem{bacry2013some}
	Emmanuel Bacry, Sylvain Delattre, Marc Hoffmann, and Jean-Fran{\c{c}}ois Muzy.
	\newblock Some limit theorems for hawkes processes and application to financial
	statistics.
	\newblock {\em Stochastic Processes and their Applications}, 123(7):2475--2499,
	2013.
	
	\bibitem{bessyroland:hal-02546343}
	Yannick Bessy-Roland, Alexandre Boumezoued, and Caroline Hillairet.
	\newblock {Multivariate Hawkes process for cyber insurance}.
	\newblock working paper or preprint, April 2020.
	
	\bibitem{embrechts_liniger_lin_2011}
	Paul Embrechts, Thomas Liniger, and Lu~Lin.
	\newblock Multivariate hawkes processes: an application to financial data.
	\newblock {\em Journal of Applied Probability}, 48(A):367–378, 2011.
	
	\bibitem{errais2010affine}
	Eymen Errais, Kay Giesecke, and Lisa~R Goldberg.
	\newblock Affine point processes and portfolio credit risk.
	\newblock {\em SIAM Journal on Financial Mathematics}, 1(1):642--665, 2010.
	

	
	\bibitem{hawkes1971spectra}
	Alan~G Hawkes.
	\newblock Spectra of some self-exciting and mutually exciting point processes.
	\newblock {\em Biometrika}, 58(1):83--90, 1971.
	
	\bibitem{hillairet2021malliavinstein}
	Caroline Hillairet, Lorick Huang, Mahmoud Khabou, and Anthony R\'{e}veillac.
	\newblock The Malliavin-Stein method for Hawkes functionals, 2021.
	
	\bibitem{nourdin:hal-01314406}
	Ivan Nourdin and Giovanni Peccati.
	\newblock {\em {Normal approximations with Malliavin calculus}}, volume 192 of
	{\em Cambridge Tracts in Mathematics}.
	\newblock {Cambridge University Press}, May 2012.
	
	\bibitem{nourdin2010multivariate}
	Ivan Nourdin, Giovanni Peccati, and Anthony R{\'e}veillac.
	\newblock Multivariate normal approximation using stein's method and malliavin
	calculus.
	\newblock In {\em Annales de l'IHP Probabilit{\'e}s et statistiques},
	volume~46, pages 45--58, 2010.
	
	\bibitem{locherbach2017large}
	Eva~L\" ocherbach.
	\newblock Large deviations for cascades of diffusions arising in oscillating
	systems of interacting hawkes processes, 2017.
	
	\bibitem{reynaudbouret:hal-00866823}
	Patricia Reynaud-Bouret, Vincent Rivoirard, and Christine Tuleau-Malot.
	\newblock {Inference of functional connectivity in Neurosciences via Hawkes
		processes}.
	\newblock In {\em {1st IEEE Global Conference on Signal and Information
			Processing}}, Austin, United States, December 2013.
			
		\bibitem{giovanni2010multi}
    Giovanni	Peccati , Cengbo Zheng, et~al.
	\newblock Multi-dimensional gaussian fluctuations on the poisson space.
	\newblock {\em Electronic Journal of Probability}, 15:1487--1527, 2010.		
			
	\bibitem{picard1996formules}
	Jean Picard.
	\newblock Formules de dualit{\'e} sur l'espace de poisson.
	\newblock In {\em Annales de l'IHP Probabilit{\'e}s et statistiques},
	volume~32, pages 509--548, 1996.
	
\end{thebibliography}
\end{document}